\documentclass[11pt, a4paper]{article}

\usepackage{typearea}
\typearea{12}

\usepackage{amsmath}
\usepackage{amsthm}
\usepackage{amssymb}
\usepackage{color}
\usepackage{url}
\usepackage{graphicx}
\usepackage{ascmac}
\usepackage{float}
\usepackage{comment}
\usepackage{ifthen}
\usepackage{mathdots}
\usepackage{algorithm}
\usepackage{algorithmic}

\theoremstyle{definition}
\newtheorem{thm}{Theorem}

\newtheorem{lem}[thm]{Lemma}

\newtheorem{rem}{Remark}

\allowdisplaybreaks[4]

\numberwithin{equation}{section}
\numberwithin{thm}{section}
\numberwithin{rem}{section}

\numberwithin{thm_ja}{section}
\numberwithin{rem_ja}{section}

\title{
Kernel quadrature by applying a point-wise gradient descent method to discrete energies
}
\author{Ken'ichiro Tanaka% 
\footnote{
Department of Mathematical Informatics, 
Graduate School of Information Science and Technology,
University of Tokyo. 
7-3-1 Hongo, Bunkyo-ku, Tokyo, 113-8656, Japan. 
e-mail: \texttt{kenichiro@mist.i.u-tokyo.ac.jp}
} 
\footnote{
PRESTO, Japan Science and Technological Agency (JST),
4-1-8 Honcho, Kawaguchi-shi, Saitama, 332-0012, Japan.
}
}
\date{\today}
\begin{document}

\maketitle

\begin{abstract}
We propose a method for generating nodes for kernel quadrature by 
a point-wise gradient descent method. 
For kernel quadrature, 
most methods for generating nodes
are based on the worst case error of a quadrature formula in a reproducing kernel Hilbert space corresponding to the kernel. 
In typical ones among those methods, 
a new node is chosen among a candidate set of points in each step 
by an optimization problem with respect to a new node. 
Although such sequential methods are appropriate for adaptive quadrature, 
it is difficult to apply standard routines for mathematical optimization to the problem.
In this paper, 
we propose a method 
that updates a set of points one by one 
with a simple gradient descent method. 
To this end, 
we provide an upper bound of the worst case error by 
using the fundamental solution of the Laplacian on $\mathbf{R}^{d}$. 
We observe the good performance of the proposed method by numerical experiments. 

\bigskip


\bigskip


\end{abstract}

%--------------------
\section{Introduction}

This paper is concerned with kernel quadrature, 
an approach to deriving numerical integration formulas with kernels. 
Let $d \in \mathbf{Z}_{+}$ be a positive integer, 
$\Omega \subset \mathbf{R}^{d}$ a region with an non-empty interior in $\mathbf{R}^{d}$, 
and $\mu$ a Borel measure on $\Omega$.
A formula for quadrature is formally expressed as
\begin{align}
\int_{\Omega} f(x) \, \mathrm{d}\mu(x)
\approx
\sum_{j=1}^{N} w_{j} \, f(x_{j}),
\label{eq:num_int_H_K}
\end{align}
where
$\mathcal{X}_{N} = \{x_{1}, \ldots , x_{N} \} \subset \Omega$ and 
$\mathcal{W}_{N} = \{ w_{1}, \ldots, w_{N} \} \subset \mathbf{R}$
are the sets of distinct nodes and weights, respectively. 

Besides
Monte Carlo (MC) methods and 
quasi Monte Carlo (QMC) methods, 
kernels have been used to derive numerical integration formulas recently. 
There are several categories of such kernel methods. 
We can see a category of randomized methods for choosing nodes $x_{j}$. 
It includes 
importance sampling \cite{bib:LiuLee_imp_sampl_2017}, 
random feature expansions~\cite{bib:Bach_equi_kerquad_randfeat_2017}, 
kernel quadrature with determinantal point processes (DPPs) \cite{bib:Belhadji_etal_KerDPP_2019}, etc. 

On the other hand, 
there are categories of deterministic methods. 
One of them consists of sequential algorithms choosing nodes $x_{j}$ one-by-one with greedy ways. 
Such algorithms are
%---[sequential algo.]---
kernel herding (KH)
\cite{bib:Bach_etal_herding_2012, bib:Chen_etal_kernel_herding_2010}, 
sequential Bayesian quadrature (SBQ) 
\cite{bib:HusDuven_SBQ_2012, bib:Oettershagen_PhD_2017}, 
orthogonal matching persuit (OMP)
\cite{bib:Oettershagen_PhD_2017}, 
and their variants  
\cite{bib:Briol_etal_FWBQ_2015, 
bib:Chen_etal_Stein_2018, 
bib:Joseph_etal_min_ene_2015, 
bib:Joseph_etal_min_ene_2019,
%bib:Pronzato_anal_MMD_2021, 
bib:Pronzato_etal_min_ene_2020, 
bib:Pronzato_etal_min_ene_snglr_2021, 
bib:Teymur_etal_MMD_2020}. 
%---[sequential algo.]---
For example, the SBQ is realized by the procedure
\begin{align}
x_{N+1} \in \mathop{\mathrm{argmin}}_{x \in \Omega} 
\left( 
e^{\mathrm{wor}}(\mathcal{X}_{N} \cup \{ x \}, \mathcal{W}_{N+1}^{\ast}; \mathcal{H}_{K}) 
\right)
\qquad
(n = 1,2,\ldots),
\label{eq:SBQ}
\end{align}
where $e^{\mathrm{wor}}$, $\mathcal{W}_{N+1}^{\ast}$, and $\mathcal{H}_{K}$
denote an worst case error, optimal set of weights, and reproducing kernel Hilbert space
defined later in~\eqref{eq:def_WCE}, \eqref{eq:opt_weights}, and Section~\ref{sec:BasicKQ}, respectively. 
The methods in this category are rational in that we can add new nodes with keeping existing nodes. 
Such procedures are useful for adaptive quadrature, 
which is useful in the case of a high-dimensional and/or complicated region $\Omega$.  
However, each step like~\eqref{eq:SBQ}
requires nonlinear and nonconvex optimization choosing nodes among prepared ones in the region $\Omega$. 

There is another category of methods obtaining whole nodes at once. 
For example, 
%---[approx. Fekete etc.]---
Oettershagen \cite{bib:Oettershagen_PhD_2017} proposes a method solving nonlinear equation. 
Furthermore, 
Fekete points, which maximize the determinant of a kernel matrix, 
are known to be useful as nodes for quadrature. 
Unfortunately, 
it is hard to find them exactly because the maximization of the determinant is intractable in general. 
However, 
some approximate optimization methods have been proposed 
\cite{bib:KST_appFekete_2020, bib:KTanaka_KQ_conv_opt_2020}, 
although these are aimed at approximation of functions. 
%---[approx. Fekete etc.]---
In fact, 
these optimization methods employ convex objective functions for finding a set of nodes. 
They enable us to find a good set of nodes by standard optimization routines
without preparing a candidate set of nodes in $\Omega$. 
In particular, 
a logarithmic energy with an external field
\begin{align}
I(x_1, \ldots, x_{N}) 
& = \varepsilon^2 \sum_{k=1}^{N} x_k^2 + \sum_{1 \leq i < j \leq N} \log \frac{1}{|x_i - x_j|} 
\label{eq:LogEne_OneDim}
\end{align}
is derived in \cite{bib:KST_appFekete_2020} 
as an objective function to find approximate Fekete points 
for a one-dimensional Gaussian kernel 
\(
\displaystyle
K(x, y) 
= 
\exp
\left(
- \varepsilon^{2} (x-y)^{2}
\right)
\)
with $\varepsilon > 0$. 
This tool is restricted to a one-dimensional Gaussian kernel and has the following problems: 
\begin{enumerate}
\item
It is difficult to extend it to higher-dimensional cases. 
\item
The relationship between it and the worst case error $e^{\mathrm{wor}}$ is unclear. 
\end{enumerate}
Therefore
it would be useful for kernel quadrature
if we can obtain an extension of the energy in~\eqref{eq:LogEne_OneDim}
as an upper bound of the worst case error. 
Here we note that 
finding a set of nodes at once by optimization of such an extended energy may be difficult 
when we need to find very many nodes in very high-dimensional region $\Omega$. 
However, 
in such a case, 
such an energy can be used also for a sequential algorithm finding nodes one-by-one. 

In this paper, 
we provide an extension of the energy in~\eqref{eq:LogEne_OneDim}
by using the fundamental solution of the Laplacian on $\mathbf{R}^{d}$
and show that it provides an upper bound of the worst case error. 
A method for deriving these results is based on the idea of \cite{bib:Steinerberger_LogEneS2_2020}.
Based on the bound, 
we propose a method 
that updates a set of points one by one 
with a simple gradient descent method, 
which we call a point-wise gradient descent method. 
This method does not require a candidate set of points in $\Omega$
and relatively easy to be implemented. 

The rest of this paper is organized as follows. 
In Section \ref{sec:BasicKQ}, 
we review some basic notions for kernel quadrature. 
In Section~\ref{sec:1D_appFekete_Gauss}, 
we review the basics about the energy in~\eqref{eq:LogEne_OneDim}. 
In Section~\ref{sec:bd_of_wce_by_fund_sol}, 
we derive an extension of the energy by 
using the fundamental solution of the Laplacian on $\mathbf{R}^{d}$. 
In Section~\ref{sec:approx_min_wce}, 
we propose the point-wise gradient descent method. 
In Section~\ref{sec:num_exper}, 
we show some results of numerical experiments. 
In Section~\ref{sec:concl}, 
we conclude this paper. 

%--------------------
\section{Basic notions for kernel quadrature}
\label{sec:BasicKQ}

Let $K: \Omega \times \Omega \to \mathbf{R}$ be a continuous and symmetric function 
and assume that it is positive definite. 
Let $\mathcal{H}_{K}(\Omega)$ be the reproducing kernel Hilbert space (RKHS) corresponding to the kernel $K$. 
It has the following properties:
\begin{enumerate}
\item
\(
\forall x \in \Omega, \ 
K(\, \cdot, \, x) \in \mathcal{H}_{K}(\Omega), 
\)

\item
\(
\forall f \in \mathcal{H}_{K}(\Omega), \ 
\forall x \in \Omega, \ 
\langle f, \, K(\, \cdot, \, x) \rangle_{\mathcal{H}_{K}} = f(x).
\)

\end{enumerate}
The latter is called a reproducing property. 
We consider quadrature formula~\eqref{eq:num_int_H_K} for a function $f \in \mathcal{H}_{K}(\Omega)$. 
This approximation can be regarded as that of the measure $\mu$ as follows:
\begin{align}
\mu \approx \sum_{j=1}^{N} w_{j} \, \delta_{x_{j}}.
\end{align}
For the formula in~\eqref{eq:num_int_H_K}, 
we can define the worst-case error $e^{\mathrm{wor}}(\mathcal{X}_{N}, \mathcal{W}_{N}; \mathcal{H}_{K})$ by 
\begin{align}
e^{\mathrm{wor}}(\mathcal{X}_{N}, \mathcal{W}_{N}; \mathcal{H}_{K})
=
\sup_{
\begin{subarray}{c}
f \in \mathcal{H}_{K} \\
\| f \|_{\mathcal{H}_{K}} \leq 1
\end{subarray}
} \left|
\int_{\Omega} f(x) \, \mathrm{d}x 
- 
\sum_{i = 1}^{N} w_{i} \, f(x_{i})
\right|.
\label{eq:def_WCE}
\end{align}
In general, it is desirable to construct a good point set $\mathcal{X}_{N}$ and weight set $\mathcal{W}_{N}$ 
that make the worst-case error as small as possible. 
To approach this goal, 
the following well-known expression of the worst-case error is useful. 
\begin{align}
& (e^{\mathrm{wor}}(\mathcal{X}_{N}, \mathcal{W}_{N}; \mathcal{H}_{K}))^{2} 
\notag \\
& = \left \|
\int_{\Omega} K(y, x) \, \mathrm{d}\mu(x)
-
\sum_{j=1}^{N} w_{j} \, K(y, x_{j})
\right \|_{\mathcal{H}_{K}}^{2}
\notag \\
& =
\int_{\Omega} \int_{\Omega} K(y,z) \, \mathrm{d}\mu(y) \mathrm{d}\mu(z)
- 2 \sum_{j=1}^{N} w_{j} \int_{\Omega} K(x_{j}, x) \, \mathrm{d}\mu(x)
+ \sum_{i=1}^{N} \sum_{j=1}^{N} w_{i} w_{j} \, K(x_{i}, x_{j}).
\label{eq:WCE}
\end{align}
This is owing to the reproducing property
$\langle f, \, K( \cdot, \, x) \rangle_{\mathcal{H}_{K}} = f(x)$. 
Therefore 
sets $\mathcal{X}_{N}$ and $\mathcal{W}_{N}$ minimizing the value
\begin{align}
- 2 \sum_{j=1}^{N} w_{j} \int_{\Omega} K(x_{j}, x) \, \mathrm{d}\mu(x)
+ \sum_{i=1}^{N} \sum_{j=1}^{N} w_{i} w_{j} \, K(x_{i}, x_{j})
\label{eq:WCE_two_terms}
\end{align}
are required as a quadrature formula. 
We often assume that 
\begin{align}
w_{i} = \frac{1}{N}
\qquad
(i = 1,\ldots, N)
\label{eq:equi_weights}
\end{align}
for simplicity. 

%----------
\section{Approximate Fekete points for Gaussian kernels}
\label{sec:1D_appFekete_Gauss}

\subsection{Expression of the worst case error by determinants}

If we fix the point set $\mathcal{X}_{N}$, the right hand side in~\eqref{eq:WCE}
becomes a quadratic form with respect to the weight set $\mathcal{W}_{N}$.  
Therefore, we can easily find its minimizer $\mathcal{W}_{N}^{\ast} = \{ w_{i}^{\ast} \}$ as follows:
\begin{align}
\boldsymbol{w}^{\ast} = 
\int_{\Omega}
\mathcal{K}_{\mathcal{X}_{N}}^{-1}\, \boldsymbol{k}_{\mathcal{X}_{N}}(x)
\, \mathrm{d}x, 
\label{eq:opt_weights}
\end{align}
where
\(
\boldsymbol{w}^{\ast}
=
(w_{1}^{\ast}, \ldots , w_{N}^{\ast})^{T}
\), and 
\begin{align}
\mathcal{K}_{\mathcal{X}_{N}}
=  
\begin{bmatrix}
K(x_{1}, x_{1}) & K(x_{1}, x_{2}) & \cdots & K(x_{1}, x_{N}) \\
K(x_{2}, x_{1}) & K(x_{2}, x_{2}) & \cdots & K(x_{2}, x_{N}) \\
\vdots & \vdots & \ddots & \vdots \\
K(x_{N}, x_{1}) & K(x_{N}, x_{2}) & \cdots & K(x_{N}, x_{N}) \\
\end{bmatrix},
\quad
\boldsymbol{k}_{\mathcal{X}_{N}}(x)
=
\begin{bmatrix}
K(x, x_{1}) \\ K(x, x_{2}) \\ \vdots \\ K(x, x_{N})
\end{bmatrix}.
\label{eq:A_and_b}
\end{align}
By using these, 
we can express the worst-case error with the optimal weights $\mathcal{W}_{N}^{\ast}$ by
\begin{align}
(e^{\mathrm{wor}}(\mathcal{X}_{N}, \mathcal{W}_{N}^{\ast}; \mathcal{H}_{K}))^{2}
=
\frac{1}{\det \mathcal{K}_{\mathcal{X}_{N}}} \, 
\det \left[
\begin{array}{c|ccc}
k_{0} & k_{1} & \cdots & k_{N} \\
\hline
k_{1} & & & \\
\vdots & & \mathcal{K}_{\mathcal{X}_{N}} & \\
k_{N} & & & 
\end{array}
\right], 
\label{eq:wce_optW}
\end{align}
where
\begin{align}
k_{0} = \int_{\Omega} \int_{\Omega} K(x, y) \, \mathrm{d}x \mathrm{d}y, \qquad 
k_{i} = \int_{\Omega} K(x,x_{i}) \, \mathrm{d}x \quad (i = 1,\ldots , n). 
\label{eq:K_integrals}
\end{align}
Clearly, the value $e^{\mathrm{wor}}(\mathcal{X}_{N}, \mathcal{W}_{N}^{\ast}; \mathcal{H}_{K})$
is less than or equal to the worst-case error with the equal weights in \eqref{eq:equi_weights}: 
\begin{align}
e^{\mathrm{wor}}(\mathcal{X}_{N}, \mathcal{W}_{N}^{\ast}; \mathcal{H}_{K})
\leq 
e^{\mathrm{wor}}(\mathcal{X}_{N}, \{ 1/N \}_{i=1}^{N}; \mathcal{H}_{K}).
\label{eq:opt_w_leq_equi_w}
\end{align}

According to Formula~\eqref{eq:wce_optW}%
\footnote{Although its derivation is fundamental, we write it in Appendix~\ref{sec:wce_optW} for readers' convenience.}, 
maximization of the determinant of the matrix $\mathcal{K}_{\mathcal{X}_{N}}$ seems to be useful, 
although we do not have some reasonable estimate of the worst case error for its maximizer: the Fekete points. 
Unfortunately, the maximization of $\det \mathcal{K}_{\mathcal{X}_{N}}$ is not tractable in general. 
However, 
in the case of the one-dimensional Gaussian kernel $K$, 
we can obtain a tractable approximation of the determinant by expanding the kernel \cite{bib:KST_appFekete_2020}.

%----------
\subsection{Expansion and truncation of the one-dimensional Gaussian kernel}

Let $K(x,y) := \exp(-\varepsilon^{2} |x - y|^{2})$ be the one-dimensional Gaussian kernel. 
Then, we have 
\begin{equation*}
  \begin{split}
  K(x,y)
  \approx
\widehat{K}(x,y) 
  &= \sum_{\ell=0}^{n-1} \varphi_\ell(x) \varphi_\ell(y), 
  \end{split}
\end{equation*}
where 
\(
\displaystyle
\varphi_\ell(x) 
:= 
\exp(-\varepsilon^2 x^2) 
\sqrt{\frac{2^\ell \varepsilon^{2\ell}}{\ell!}} x^\ell
\). 
Let 
$\widehat{\mathcal{K}}_{\mathcal{X}_N} = (\widehat{K}(x_k,x_m))_{k,m=1}^{N} \in \mathbf{R}^{n \times n}$ 
be the corresponding kernel matrix, 
where $\mathcal{X}_N = \{ x_{1}, \ldots. x_{N} \}$. 
Then, it is shown in \cite{bib:KST_appFekete_2020} that 
\begin{align}
-\log \big( \det \widehat{\mathcal{K}}_{\mathcal{X}_{N}} \big)^{1/2} 
=
C_{\varepsilon, N} + \varepsilon^2 \sum_{k=1}^{N} x_k^2 + \sum_{1 \leq i < j \leq N} \log \frac{1}{|x_i - x_j|}, 
\end{align}
where $C_{\varepsilon, N}$ is a constant independent of $\mathcal{X}_{N}$. 
This is the logarithmic energy with an external field shown in~\eqref{eq:LogEne_OneDim}, 
for which we need to find a minimizer. 
We can show that this energy is convex and that there exists a unique minimiser 
$\mathcal{X}_{N}^* \in \mathcal{R}_N := = \{ (x_1, \ldots, x_{N}) \in \Omega^{N} \mid  x_1 < x_2 < \cdots < x_{N-1} < x_{N} \}$.
Therefore
we can find $\mathcal{X}_{N}^{\ast}$ numerically by using a standard optimization technique like the Newton method. 

Unfortunately, 
extension of the above argument to higher-dimensional cases is difficult 
because of the complicated structure of the kernel matrix of the truncated kernel $\widehat{K}$ for $d \geq 2$. 
Therefore 
we introduce another way for such extension in Section~\ref{sec:bd_of_wce_by_fund_sol} below. 

%--------------------
\section{Bound for a discrete energy of the Gaussian kernel with the fundamental solution of the Laplacian}
\label{sec:bd_of_wce_by_fund_sol}

Taking the relation in~\eqref{eq:opt_w_leq_equi_w} into account, 
we focus on the case of the equal weights in this section. 

%----------
\subsection{Discrete energy given by an integral of the heat kernel}

Let $d \geq 2$. 
We use the $d$-dimensional fundamental solution $G_{d}(x,y)$
for the Laplacian $\varDelta$ on $\mathbf{R}^{d}$. 
It is given by the following expression: 
\begin{align}
G_{d}(x,y) =
\begin{cases}
%\dfrac{1}{2} |x - y| & (d = 1), \\[6pt]
\dfrac{1}{2\pi} \log \| x - y\| & (d = 2), \\[6pt]
- \dfrac{1}{2 (d-2) s_{d}} \, \dfrac{1}{\| x - y \|^{d-2}} & (d \geq 3),
\end{cases}
\label{eq:fund_sol_Laplace}
\end{align}
where $s_{d}$ is the surface area of the $(d-1)$ dimensional unit sphere. 
The fundamental solution satisfies
\begin{align}
\varDelta_{x} G_{d}(x,y) = \delta_{y}(x) \quad (= \delta(x - y)).
\label{eq:Laplace_Green_delta}
\end{align}
The following lemmas are based on the ideas in \cite{bib:Steinerberger_LogEneS2_2020}, 
whereas their proofs are slightly different from those in the paper. 
Their proofs are provided in Section~\ref{sec:proofs}.

\begin{lem}
\label{lem:Green_exps_delta_expt_delta}
Let $s$ and $t$ be positive real numbers
and 
let $a$ and $b$ be points in $\mathbf{R}^{d}$.
Then the following equality holds:
\begin{align}
\int_{\mathbf{R}^{d}}
\mathrm{d}x
\int_{\mathbf{R}^{d}}
\mathrm{d}y \
G_{d}(x,y) 
\, \mathrm{e}^{s \varDelta_{x}} \delta_{a}(x)
\, \mathrm{e}^{t \varDelta_{y}} \delta_{b}(y)
=
\int_{\mathbf{R}^{d}}
\mathrm{d}y \
G_{d}(a,y) 
\, \mathrm{e}^{(s+t) \varDelta_{y}} \delta_{b}(y).
\notag
\end{align}
\end{lem}

\begin{lem}
\label{lem:Green_expt_delta_eq}
Let $t$ be a positive real number
and 
let $b$ be a point in $\mathbf{R}^{d}$.
Then, the value
\begin{align}
\int_{\mathbf{R}^{d}}
\mathrm{d}y \
G_{d}(b,y) 
\, \mathrm{e}^{t \varDelta_{y}} \delta_{b}(y)
\notag
\end{align}
is bounded and depends only on $d$ and $t$. 
\end{lem}

\begin{lem}
\label{lem:Green_expt_delta_disj}
Let $t$ be a positive real number
and 
let $a$ and $b$ be disjoint points in $\mathbf{R}^{d}$.
Then the following equality holds:
\begin{align}
\int_{\mathbf{R}^{d}}
\mathrm{d}y \
G_{d}(a,y) 
\, \mathrm{e}^{t \varDelta_{y}} \delta_{b}(y)
=
G_{d}(a,b) + \int_{0}^{t} \frac{1}{(4\pi s)^{d/2}} \, \exp \left( -\frac{\| a - b \|^{2}}{4s} \right) \, \mathrm{d}s.
\notag
\end{align}
\end{lem}

For a discrete measure
\begin{align}
\mu_{N} = \frac{1}{N} \sum_{j=1}^{N} \delta_{x_{j}}
\label{eq:def_disc_meas}
\end{align}
with disjoint sets $\{ x_{i} \}_{i=1}^{N}$, 
we introduce a renormalized energy. 
To this end, 
for $\mu_{N}$ and a positive number $t > 0$, 
we define $A_{d}(t, \mu_{N})$ by
\begin{align}
A_{d}(t, \mu_{N})
:= 
\int_{\mathbf{R}^{d}}
\mathrm{d}x
\int_{\mathbf{R}^{d}}
\mathrm{d}y \
G_{d}(x,y) 
\, \mathrm{e}^{\frac{t}{2} \varDelta_{x}} \mu_{N}(x)
\, \mathrm{e}^{\frac{t}{2} \varDelta_{y}} \mu_{N}(y). 
\label{eq:def_A_d_t_mu}
\end{align}
Furthermore, we define a constant:
\begin{align}
C_{d}(t)
:=
\int_{\mathbf{R}^{d}}
\mathrm{d}y \
G_{d}(0,y) 
\, \mathrm{e}^{t \varDelta_{y}} \delta_{0}(y).
\label{eq:def_C_d_t}
\end{align}
Then, 
we can derive the following relation 
by using the above lemmas. 

\begin{thm}
\label{thm:int_heat_kernel_ene}
Let $t > 0$ be a positive number. 
Then, we have
\begin{align}
\frac{1}{N^{2}} \sum_{i \neq j} 
\int_{0}^{t} \frac{1}{(4\pi s)^{d/2}} \, \exp \left( -\frac{\| x_{i} - x_{j} \|^{2}}{4s} \right) \, \mathrm{d}s
= 
A_{d}(t, \mu_{N}) - \frac{C_{d}(t)}{N} - \frac{1}{N^{2}} \sum_{i \neq j} G_{d}(x_{i}, x_{j}). 
\label{eq:ene_Gau_and_fund_sol}
\end{align}
\end{thm}

\begin{proof}
By using Lemmas~\ref{lem:Green_exps_delta_expt_delta}, \ref{lem:Green_expt_delta_eq}, and~\ref{lem:Green_expt_delta_disj}, 
we can derive the following equalities: 
\begin{align}
A_{d}(t, \mu_{N})
& = 
\frac{1}{N^{2}} \sum_{i=1}^{N} \sum_{j=1}^{N} 
\int_{\mathbf{R}^{d}}
\mathrm{d}x
\int_{\mathbf{R}^{d}}
\mathrm{d}y \
G_{d}(x,y) 
\, \mathrm{e}^{\frac{t}{2} \varDelta_{x}} \delta_{x_{i}}(x)
\, \mathrm{e}^{\frac{t}{2} \varDelta_{y}} \delta_{x_{j}}(y)
\notag \\
& = 
\frac{1}{N^{2}} \sum_{i=1}^{N} \sum_{j=1}^{N} 
\int_{\mathbf{R}^{d}}
\mathrm{d}y \
G_{d}(x_{i},y) 
\, \mathrm{e}^{t \varDelta_{y}} \delta_{x_{j}}(y)
\qquad (\because \text{Lemma~\ref{lem:Green_exps_delta_expt_delta}})
\notag \\
& = 
\frac{1}{N^{2}} 
\left( 
\sum_{i=j} + \sum_{i \neq j}
\right)
\int_{\mathbf{R}^{d}}
\mathrm{d}y \
G_{d}(x_{i},y) 
\, \mathrm{e}^{t \varDelta_{y}} \delta_{x_{j}}(y)
\notag \\
& = 
\frac{C_{d}(t)}{N}
+ 
\frac{1}{N^{2}} \sum_{i \neq j}
\left\{
G_{d}(x_{i}, x_{j}) + \int_{0}^{t} \frac{1}{(4\pi s)^{d/2}} \, \exp \left( -\frac{\| x_{i} - x_{j} \|^{2}}{4s} \right) \, \mathrm{d}s
\right\}
\notag \\
& \phantom{=} \ \  
(\because \text{Lemmas~\ref{lem:Green_expt_delta_eq} and~\ref{lem:Green_expt_delta_disj}}). 
\notag
\end{align}
Hence the conclusion follows.
\end{proof}

Here we present a sketch of an idea to estimate both sides of~\eqref{eq:ene_Gau_and_fund_sol}. 
Suppose that
\begin{itemize}
\item
the points $x_{j}$ are in a bounded region $\Omega \subset B[0,r] \subset \mathbf{R}^{d}$, and

\item
$t$ is sufficiently large so that $A_{d}(t,\mu_{N})$ is almost independent of $\{ x_{j} \}_{j=1}^{N}$. 

\end{itemize}
Under these conditions, we may have
\begin{align}
& \int_{0}^{t} \frac{1}{(4\pi s)^{d/2}} \, \exp \left( -\frac{\| x_{i} - x_{j} \|^{2}}{4s} \right) \, \mathrm{d}s
\notag \\
& =
\left( \int_{0}^{r} + \int_{r}^{t} \right) \frac{1}{(4\pi s)^{d/2}} \, \exp \left( -\frac{\| x_{i} - x_{j} \|^{2}}{4s} \right) \, \mathrm{d}s
\notag \\
& \approx 
\int_{0}^{r} \frac{1}{(4\pi s)^{d/2}} \, \exp \left( -\frac{\| x_{i} - x_{j} \|^{2}}{4s} \right) 
+
\int_{r}^{t} \frac{1}{(4\pi s)^{d/2}} \, \mathrm{d}s \\
& \approx
\hat{C}_{d,r} \, \exp \left( -\frac{\| x_{i} - x_{j} \|^{2}}{4r} \right) 
+
\int_{r}^{t} \frac{1}{(4\pi s)^{d/2}} \, \mathrm{d}s.
\end{align}
Therefore we may state that 
the terms depending on $\{ x_{j} \}_{j=1}^{N}$ in~\eqref{eq:ene_Gau_and_fund_sol} are:
\begin{itemize}
\item
\(
\displaystyle
\frac{1}{N^{2}} \sum_{i \neq j} \hat{C}_{d,r} \, \exp \left( -\frac{\| x_{i} - x_{j} \|^{2}}{4r} \right) 
\)
\quad (in the LHS) \ and

\item
\(
\displaystyle
- \frac{1}{N^{2}} \sum_{i \neq j}
G_{d}(x_{i}, x_{j}) 
\)
\quad (in the RHS). 

\end{itemize}

We make the above rough idea rigorous in Section~\ref{sec:bd_Gau_ene_fund_sol} below. 
That is, 
we derive a bound of the energy of the Gaussian kernel 
by using Theorem~\ref{thm:int_heat_kernel_ene}. 

%----------
\subsection{Bound of the discrete energy of the Gaussian kernel}
\label{sec:bd_Gau_ene_fund_sol}

We provide a lower bound of the integral of the heat kernel in Theorem~\ref{thm:int_heat_kernel_ene}
by the following lemma, whose proof is provided in Section~\ref{sec:proofs}.

\begin{lem}
\label{thm:ineq_int_heat_fund_sol}
Let $\Omega \subset \mathbf{R}^{d}$ be a bounded region with $D := \mathop{\mathrm{diam}} \Omega < \infty$ and
let $x, y \in \Omega$ be disjoint points in the region. 
Then, for any $t$ with $t \geq D^{2}/d$,  we have
\begin{align}
\int_{0}^{t} \frac{1}{(4\pi s)^{d/2}} \, \exp \left( -\frac{\| x - y \|^{2}}{4s} \right) \, \mathrm{d}s
\geq
\frac{1}{(4\pi)^{d/2}} 
\left[
\frac{d^{d/2-1}}{2D^{d-2}} \, \exp \left( -\frac{d \, \| x - y \|^{2}}{4D^{2}} \right) 
+
h_{d,D}(t)
\right],
\label{eq:lb_of_int_heat_kernel}
\end{align}
where 
\begin{align}
h_{d,D}(t)
:=
\begin{cases}
\displaystyle
\frac{\mathrm{e}^{-d/4}}{1-d/2}
\left(
\frac{1}{t^{d/2-1}} - \frac{d^{d/2-1}}{D^{d-2}}
\right) & (d \neq 2), \\[12pt]
\displaystyle
\mathrm{e}^{-1/2}
\log \left(
\frac{2t}{D^{2}} 
\right) &  (d = 2).
\end{cases}
\label{eq:lb_tail_of_int_heat_kernel}
\end{align}
\end{lem}

\begin{rem}
The function $h_{d,D}$ in~\eqref{eq:lb_tail_of_int_heat_kernel} satisfies $h_{d,D}(t) \geq 0$ for any $t$ with $t \geq D^{2}/d$.  
\end{rem}

By combining Theorem~\ref{thm:int_heat_kernel_ene} and Lemma~\ref{thm:ineq_int_heat_fund_sol}, 
we have the following estimate of the energy of the Gaussian kernel. 

\begin{thm}
\label{thm:Gau_ene_leq_fund_sol_ene}
Let $\Omega \subset \mathbf{R}^{d}$ be a bounded region with $D := \mathop{\mathrm{diam}} \Omega < \infty$. 
Let $\{ x_{i} \}_{i=1}^{N} \subset \Omega$ be a set of disjoint points in $\Omega$ and 
let $\mu_{N}$ be the discrete measure given by~\eqref{eq:def_disc_meas}. 
Then, for any $t$ with $t \geq D^{2}/d$ and $a$ with $a \geq \sqrt{d}/(2D)$,  we have
\begin{align}
& \frac{1}{N^{2}} \sum_{i \neq j} 
\exp \left( -a^{2} \| x_{i} - x_{j} \|^{2} \right) 
\notag \\
& \leq
\frac{2(4\pi)^{d/2}D^{d-2}}{d^{d/2-1}}
\left(
- \frac{1}{N^{2}} \sum_{i \neq j} G_{d}(x_{i}, x_{j}) + A_{d}(t, \mu_{N}) - \frac{C_{d}(t)}{N} - \frac{1}{(4\pi)^{d/2}} \frac{N-1}{N} h_{d,D}(t)
\right),
\label{eq:bd_of_Gau_ene}
\end{align}
where 
$A_{d}(t, \mu_{N})$ and $C_{d}(t)$ are given by \eqref{eq:def_A_d_t_mu} and~\eqref{eq:def_C_d_t}, respectively. 
\end{thm}

\begin{proof}
It follows from 
inequality~\eqref{eq:lb_of_int_heat_kernel} in Lemma~\ref{thm:ineq_int_heat_fund_sol} 
and equality~\eqref{eq:ene_Gau_and_fund_sol} in Theorem~\ref{thm:int_heat_kernel_ene}
that
\begin{align}
& \frac{1}{N^{2}} \sum_{i \neq j} 
\frac{1}{(4\pi)^{d/2}} 
\frac{d^{d/2-1}}{2D^{d-2}} \, \exp \left( -a^{2} \| x_{i} - x_{j} \|^{2} \right) 
\notag \\
& \leq \frac{1}{N^{2}} \sum_{i \neq j} 
\frac{1}{(4\pi)^{d/2}} 
\frac{d^{d/2-1}}{2D^{d-2}} \, \exp \left( -\frac{d \, \| x_{i} - x_{j} \|^{2}}{4D^{2}} \right) 
\notag \\
& \leq
\frac{1}{N^{2}} \sum_{i \neq j} 
\int_{0}^{t} \frac{1}{(4\pi s)^{d/2}} \, \exp \left( -\frac{\| x_{i} - x_{j} \|^{2}}{4s} \right) \, \mathrm{d}s
- 
\frac{1}{(4\pi)^{d/2}} \frac{N-1}{N} h_{d,D}(t)
\notag \\
& = 
A_{d}(t, \mu_{N}) - \frac{C_{d}(t)}{N} 
- \frac{1}{N^{2}} \sum_{i \neq j} G_{d}(x_{i}, x_{j})
- \frac{1}{(4\pi)^{d/2}} \frac{N-1}{N} h_{d,D}(t).
\notag
\end{align}
Thus we have the conclusion. 
\end{proof}

In the parenthesis of the RHS of \eqref{eq:bd_of_Gau_ene}, 
the second term $A_{d}(t, \mu_{N})$ also depends on the set $\{ x_{i} \}_{i=1}^{N}$.
However, 
we can expect that the dependence tends to disappear as $t \to \infty$. 
Therefore minimization of 
\begin{align}
- \frac{1}{N^{2}} \sum_{i \neq j} G_{d}(x_{i}, x_{j}) 
\notag
\end{align}
will provide approximate minimizer of the energy of the Gaussian kernel. 

%--------------------
\section{Approximate minimization of the worst case error by a point-wise gradient descent method}
\label{sec:approx_min_wce}

In this section, 
we present a method for generating points for quadrature
based on the arguments in Section~\ref{sec:bd_of_wce_by_fund_sol}. 
Then, 
based on the relation in~\eqref{eq:opt_w_leq_equi_w},
we compute the optimal weights by using Formula~\eqref{eq:opt_weights}
to obtain a quadrature formula. 

%----------
\subsection{Objective functions}

According to Theorem~\ref{thm:Gau_ene_leq_fund_sol_ene}, 
we can provide an upper bound of the value in~\eqref{eq:WCE_two_terms}
in the case of the Gaussian kernel 
\(
\displaystyle
K(x,y) = 
\exp \left(
- a^{2} \| x - y \|^{2}
\right)
\)
with $a \geq \sqrt{d}/(2D)$.
That is, the value
\begin{align}
& - \frac{2}{N} \sum_{j=1}^{N} \int_{\Omega} K(x_{j}, x) \, \mathrm{d}\mu(x)
- \frac{\hat{C}_{d,D}}{N^{2}} \sum_{i \neq j}G_{d}(x_{i}, x_{j})
\notag \\
& \qquad + 
\underline{\hat{C}_{d,D} 
\left( 
A_{d}(t, \mu_{N}) - \frac{C_{d}(t)}{N} - \frac{1}{(4\pi)^{d/2}} \frac{N-1}{N} h_{d,D}(t)
\right)}
\label{eq:ub_WCE_two_terms}
\end{align}
is such an upper bound, where
\begin{align}
\hat{C}_{d,D}
:=
\frac{2(4\pi)^{d/2}D^{d-2}}{d^{d/2-1}}. 
\notag
\end{align}
Since the underlined part of~\eqref{eq:ub_WCE_two_terms}
is almost independent of the set $\{ x_{j} \}_{j=1}^{N}$, 
we minimize the other part of~\eqref{eq:ub_WCE_two_terms}
to obtain an approximate minimizer of the worst case error. 
In the following, 
we deal with the Gaussian kernel $K(x,y) = \exp \left( - \| x - y \|^{2} \right)$. 
That is, we consider the case that $a = 1$ in Theorem~\ref{thm:Gau_ene_leq_fund_sol_ene}. 
Let $J_{1}(x)$ be defined by
\begin{align}
J_{1}(x) 
:= 
\int_{0}^{1} \exp(- |x - y|^{2}) \, \mathrm{d}x 
= 
\frac{\sqrt{\pi}}{2} 
\left(
\mathop{\mathrm{erf}}(1-x)
+
\mathop{\mathrm{erf}}(x)
\right).
\notag
\end{align}
In the following, we consider the two and three dimensional cases.

\subsubsection{Two dimensional case: $d = 2$}

We consider a region $\Omega = [0,1]^{2}$ and measure $\mathrm{d}\mu(x) = \mathrm{d}x$. 
Then, we have $D = \sqrt{2}$ and $\hat{C}_{d,D} = 8 \pi$, 
and we can confirm that $a = 1 \geq 1/2 = \sqrt{d}/(2D)$. 
Furthermore, we have
\begin{align}
\int_{\Omega} K(x, y) \, \mathrm{d}\mu(y)
= J_{1}(x^{(1)}) J_{1}(x^{(2)}) 
=: J_{2}(x), 
\label{eq:def_Kint_2D}
\end{align}
where $x^{(1)}$ and $x^{(2)}$ are the first and second components of $x \in [0,1]^{2}$, 
respectively. 
From these and the two-dimensional fundamental solution in~\eqref{eq:fund_sol_Laplace}, 
the objective function in~\eqref{eq:ub_WCE_two_terms} is written in the form
\begin{align}
I_{2}(x_{1}, \ldots, x_{N})
:=
- \frac{2}{N} \sum_{j=1}^{N} J_{2}(x_{j})
+ \frac{4}{N^{2}} \sum_{i \neq j} \log \frac{1}{\| x_{i} - x_{j} \|}.
\label{eq:log_obj_func_2D}
\end{align}

\subsubsection{Three dimensional case: $d = 3$}

We consider a region $\Omega = [0,1]^{3}$ and measure $\mathrm{d}\mu(x) = \mathrm{d}x$. 
Then, we have $D = \sqrt{3}$ and $\hat{C}_{d,D} = 16 \pi^{3/2}$, 
and we can confirm that $a = 1 \geq 1/2 = \sqrt{d}/(2D)$. 
Furthermore, we have
\begin{align}
\int_{\Omega} K(x, y) \, \mathrm{d}\mu(y)
= J_{1}(x^{(1)}) J_{1}(x^{(2)}) J_{1}(x^{(3)}) 
=: J_{3}(x), 
\label{eq:def_Kint_3D}
\end{align}
where $x^{(1)}$, $x^{(2)}$ and $x^{(3)}$ are the first, second and third components of $x \in [0,1]^{3}$, 
respectively. 
From these and the three-dimensional fundamental solution in~\eqref{eq:fund_sol_Laplace}, 
the objective function in~\eqref{eq:ub_WCE_two_terms} is written in the form
\begin{align}
I_{3}(x_{1}, \ldots, x_{N})
:=
- \frac{2}{N} \sum_{j=1}^{N} J_{3}(x_{j})
+ \frac{2\sqrt{\pi}}{N^{2}} \sum_{i \neq j} \frac{1}{\| x_{i} - x_{j} \|}.
\label{eq:log_obj_func_3D}
\end{align}

\begin{rem}
The functions $I_{2}$ and $I_{3}$ are similar to the Riesz energies 
(see e.g.~\cite{bib:BrauGrab_PointSp_2015}), 
which have an intrinsic repelling property. 
As shown in Section~\ref{sec:num_exper}, 
this property seems to be well-suited to 
a method introduced in Section~\ref{sec:PWGD}. 
Similar energies are considered and different algorithms are applied to them in 
\cite{bib:Joseph_etal_min_ene_2015, bib:Joseph_etal_min_ene_2019}.
\end{rem}

\subsection{Regularization terms}

In the functions $I_{2}$ and $I_{3}$, 
the sums of $J_{2}(x_{j})$ and $J_{3}(x_{j})$ take roles as regularization terms, respectively. 
However, 
it was observed that their effects were so weak that 
minimization of $I_{2}$ and $I_{3}$ made points $x_{j}$ accumulate near the boundary of $\Omega$. 
Since such distribution is not appropriate for quadrature on $\Omega$, 
we introduce stronger regularization terms. 

To this end, 
we begin with approximating the characteristic functions of the regions $[0,1]^{2}$ and $[0,1]^{3}$:
\begin{align}
\delta_{[0,1]^{d}}(x) 
:=
\begin{cases}
0 & (x \in [0,1]^{d}), \\
\infty & (x \not \in [0,1]^{d}),
\end{cases}
\qquad (d = 2,3).
\notag
\end{align}
There are various choices about such approximate functions. 
In this paper, 
by using a real number $M>0$, we choose 
\begin{align}
& \tilde{\delta}_{M}^{(2)}(x) 
:=
\sum_{\ell =1}^{2} \left( \log \frac{1}{ x^{(\ell)} - M } + \log \frac{1}{ 1 + M - x^{(\ell)} } \right)
\qquad \text{and}
\notag \\
& \tilde{\delta}_{M}^{(3)}(x) 
:=
\sum_{\ell =1}^{3} \left( \frac{1}{ x^{(\ell)} - M } + \frac{1}{ 1 + M - x^{(\ell)} } \right)
\notag
\end{align}
for $\delta_{[0,1]^{2}}(x)$ and $\delta_{[0,1]^{3}}(x)$, respectively. 
The number $M$ sets a margin of the boundaries of the regions. 
By using these, we introduce
\begin{align}
& R_{2}(x_{1}, \ldots, x_{N}) :=
\frac{1}{N^{P}} \sum_{i=1}^{N} \tilde{\delta}_{M}^{(2)}(x_{i}) 
\qquad \text{and}
\label{eq:reg_term_dim2} \\
& R_{3}(x_{1}, \ldots, x_{N}) :=
\frac{1}{N^{P}} \sum_{i=1}^{N} \tilde{\delta}_{M}^{(3)}(x_{i})
\label{eq:reg_term_dim3}
\end{align}
as regularization terms in the cases of $d=2$ and $d=3$, respectively. 
The parameter $P$ determines strength of these terms. 

We generate a set $\{ x_{j} \}$ of points by minimizing the functions 
\begin{align}
I_{d}(x_{1}, \ldots, x_{N}) + R_{d}(x_{1}, \ldots, x_{N}) 
\qquad
(d = 2,3).
\notag
\end{align}
The hyper-parameters $P$ and $M$ are chosen
so that good distribution of points are obtained by the algorithm proposed below in Section~\ref{sec:PWGD}. 

\begin{rem}
The functions $\tilde{\delta}_{M}^{(2)}$ and $\tilde{\delta}_{M}^{(3)}$ are not based on any theory, 
although they imitates the corresponding fundamental solutions. 
In addition, 
the factor $N^{-P}$ in~\eqref{eq:reg_term_dim2} and~\eqref{eq:reg_term_dim3} 
may have room for improvement. 
Finding more appropriate regularization terms based on a theory is a topic for future work.
\end{rem}

\begin{rem}
In \cite{bib:Lu_etal_quad_heat_2020}, 
the authors use the heat kernels, 
which are time-dependent Gaussian kernels in a special case, 
to generate points on compact manifolds
via simulated annealing. 
They show superiority of the points given by the heat kernels 
over those given by the Riesz kernels and other QMC sequences. 
In Section~\ref{sec:num_exper} of this paper,
we observe that the functions $I_{d}$ with the regularization terms $R_{d}$
can be superior to the functions of the worst case error for the Gaussian kernel. 
\end{rem}

\subsection{Point-wise gradient descent method (PWGD)}
\label{sec:PWGD}

Here
we propose an algorithm for generating a set $\{ x_{j} \}$ of points by using the function $I_{d} + R_{d}$. 
We intend to realize a so simple algorithm with cheap computational cost
that it can be applied to high-dimensional cases with many points. 
To this end, 
it is better to avoid preparation of many candidate points in $\Omega$ 
among which points for quadrature are selected. 

Taking these considerations into account, 
we propose Algorithm~\ref{alg:PWGD} for generating points for quadrature. 
We call it a point-wise gradient descent method (PWGD). 
After preparing an initial set $\{ x_{j} \}_{j=1}^{N} \subset \Omega$, 
the algorithm updates its member one by one with a gradient of the function $I_{d} + R_{d}$ with respect to the member. 

\begin{algorithm}[ht]
\caption{Point-wise gradient descent method (PWGD)}
\label{alg:PWGD}
\begin{algorithmic}[1]
\REQUIRE a number $N$ of points, a default step size $\gamma > 0$
\ENSURE a set $\{ x_{i} \}_{i=1}^{N} \subset \Omega$ of points
\STATE generate an initial set $\{ x_{i} \}_{i=1}^{N} \subset \Omega$ randomly
\FOR {$k=1,\ldots, K_{\max}$}
	\FOR {$i=1,\ldots, N$}
		\STATE $ \displaystyle g_{i} := \nabla_{x_{i}} \left\{ I_{d}(x_{1},\ldots, x_{N}) + R_{d}(x_{1},\ldots, x_{N}) \right\} $	
		\STATE $\gamma' := \max \{ \beta \geq 0 \mid x_{i} - \beta g_{i} \in \Omega \}$
		\STATE $\gamma \leftarrow \max \{ \gamma, \gamma' \}$
		\STATE $ \displaystyle x_{i} \leftarrow x_{i} - \gamma g_{i}$ 	
	\ENDFOR
	\IF {$\displaystyle \max_{1 \leq i \leq N} \| g_{i} \| < \epsilon$} 
		\STATE \textbf{break}
	\ENDIF
\ENDFOR
\RETURN  {$\{ x_{i} \}_{i=1}^{N}$}
\end{algorithmic} 
\end{algorithm}

Finally, we obtain a set $\{ (x_{j}, w_{j}^{\ast}) \}_{j=1}^{N}$ of points and weights for quadrature by the following procedure. 
\begin{enumerate}
\item
Obtain a set $\{ x_{j} \}_{j=1}^{N} \subset \Omega$ by Algorithm~\ref{alg:PWGD} (PWGD) with equal weights: $w_{j} = 1/N$.

\item 
Compute the optimal weights $\{ w_{j}^{\ast} \}_{j=1}^{N}$ by Formula~\eqref{eq:wce_optW}:
\(
\displaystyle
\boldsymbol{w}^{\ast} = 
\int_{\Omega}
\mathcal{K}_{\mathcal{X}_{N}}^{-1}\, \boldsymbol{k}_{\mathcal{X}_{N}}(x)
\, \mathrm{d}\mu(x)
\).

\end{enumerate}

%--------------------
\section{Numerical experiments}
\label{sec:num_exper}

We compute the sets $\{ (x_{j}, w_{j}^{\ast}) \}_{j=1}^{N}$ by the proposed procedure
in the cases that $K(x,y) = \exp(- \| x - y \|^{2})$ and $\Omega = [0,1]^{d}$ for $d = 2, 3$.
We set $\gamma = 1$ in Algorithm~\ref{alg:PWGD}
and choose the hyper-parameters $P$ and $M$ experimentally.  
For comparison, 
we also use other methods shown below:
\begin{itemize}
\item[\textbf{M1}]
Sequential Bayesian quadrature (SBQ), 

\item[\textbf{M2}]
Application of the point-wise gradient descent method 
to the worst-case error with the equal weights: 
$e^{\mathrm{wor}}(\mathcal{X}_{N}, \{ 1/N \}_{i=1}^{N}; \mathcal{H}_{K})$. 

\end{itemize}
In particular, 
we deal with the latter method 
to confirm the effect of the proposed method. 
We set $\gamma = 0.1$ in Algorithm~\ref{alg:PWGD}%
\footnote{
The default step size $\gamma$ in Algorithm~\ref{alg:PWGD} 
is determined experimentally. 
In most cases, 
we observed that the inner iteration was terminated 
before it reached the maximum number $K_{\max}$ of the iteration. 
}. 
In addition, 
we used $\varepsilon = 10^{-5}$ and $\varepsilon = 10^{-4}$
in Algorithm~\ref{alg:PWGD} 
in the cases of $d = 2$ and $d = 3$, respectively. 

MATLAB programs are used for all computation in this section. 
In addition, the computation is done with the double precision floating point numbers.
The programs used for the computation are available on the web page
\cite{bib:KTanaka_Matlab_2021}.

%-----
\subsection{Two dimensional case: $d = 2$}

First, we show the generated points in
Figures~\ref{fig:2D_50p_SBQ}--\ref{fig:2D_50p_PWGD_FS_060035}.
We can observe that 
the points given by the proposed procedure 
are separated each other and do not gather around a certain point
as opposed to those given by the other methods. 

\begin{figure}[H]
\begin{minipage}[t]{0.49\linewidth}
\includegraphics[width = \linewidth]{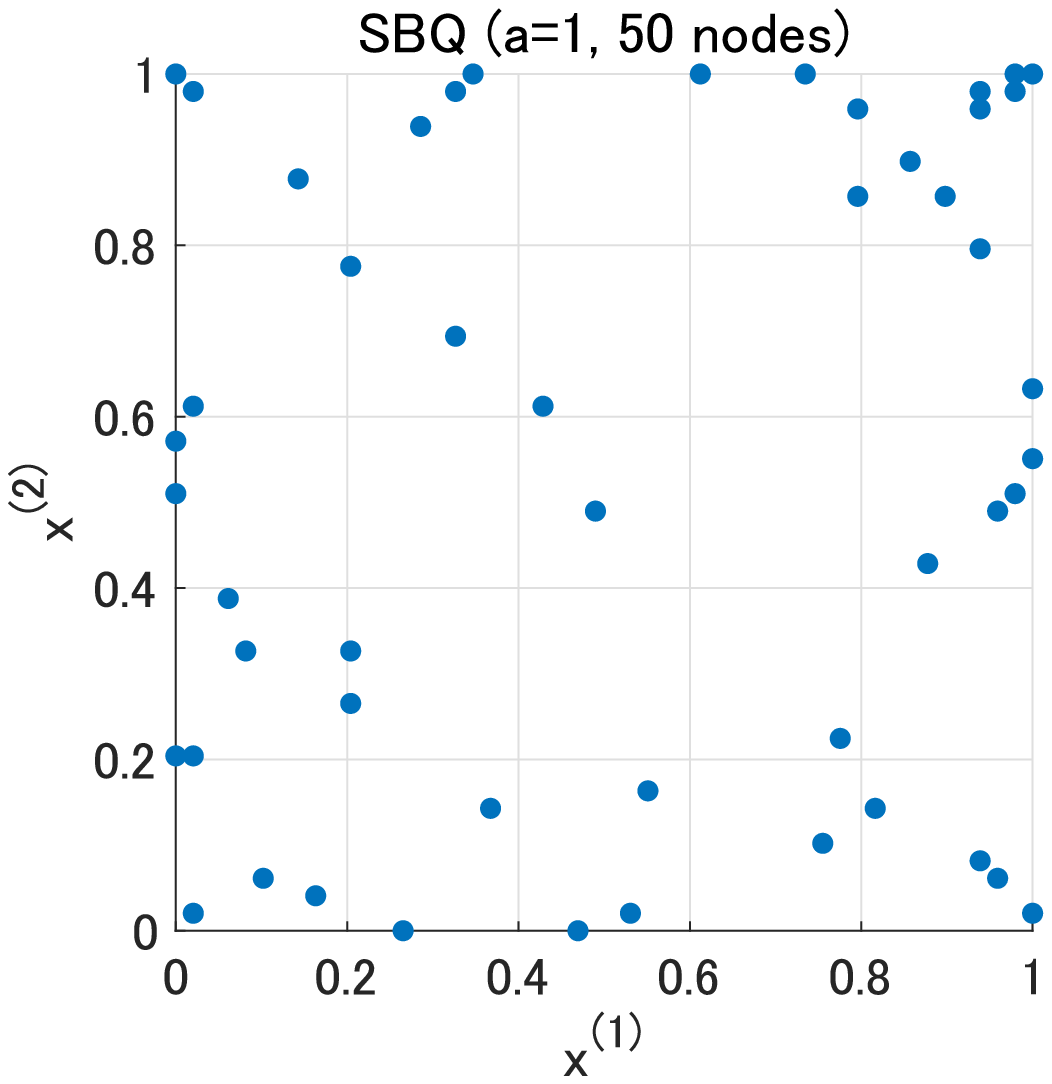}
\caption{$50$ points given by M1 (SBQ).}
\label{fig:2D_50p_SBQ}
\end{minipage}
\quad
\begin{minipage}[t]{0.49\linewidth}
\includegraphics[width = \linewidth]{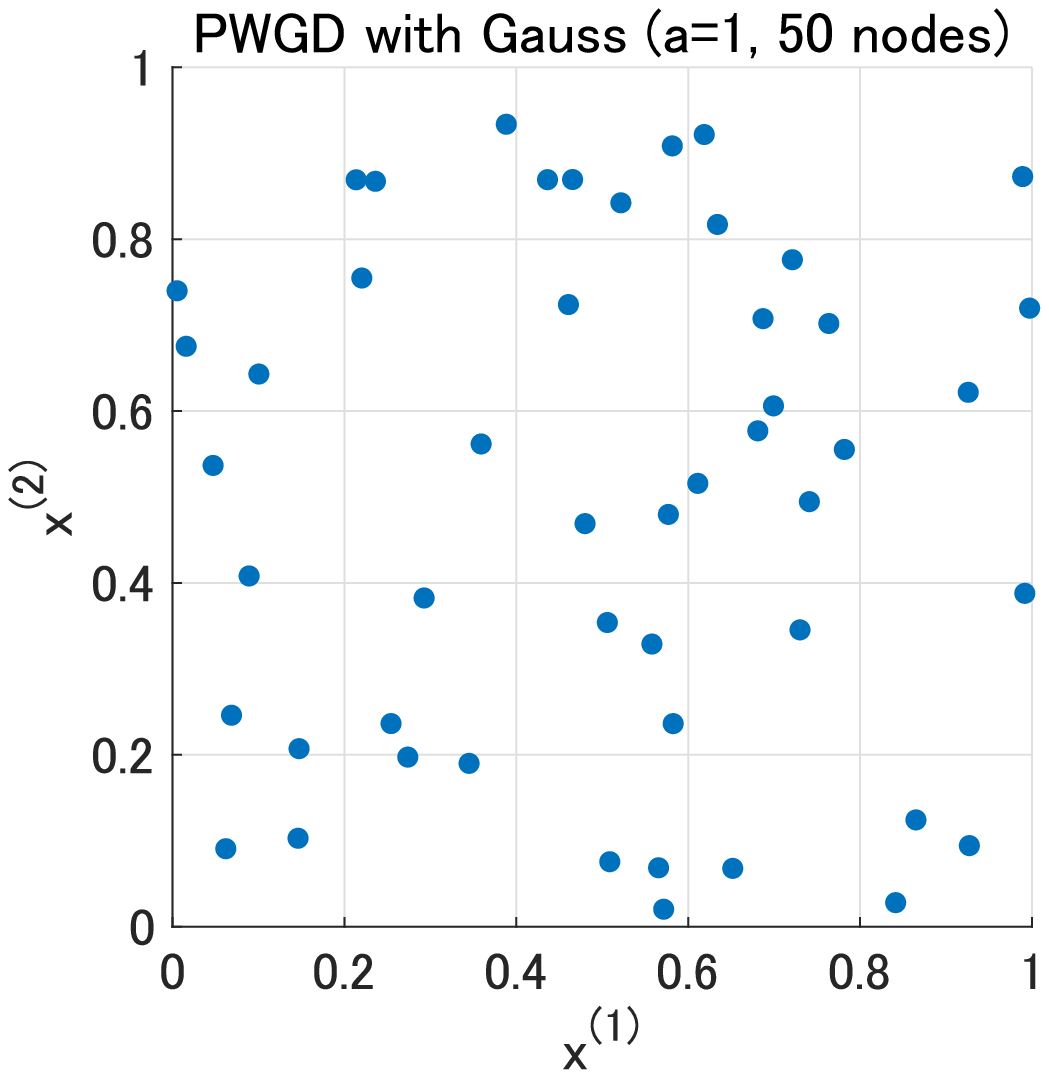}
\caption{$50$ points given by M2 (PWGD for the original worst-case error).}
\label{fig:2D_50p_PWGD_Gauss}
\end{minipage}
\end{figure}

\begin{figure}[H]
\begin{minipage}[t]{0.49\linewidth}
\includegraphics[width = \linewidth]{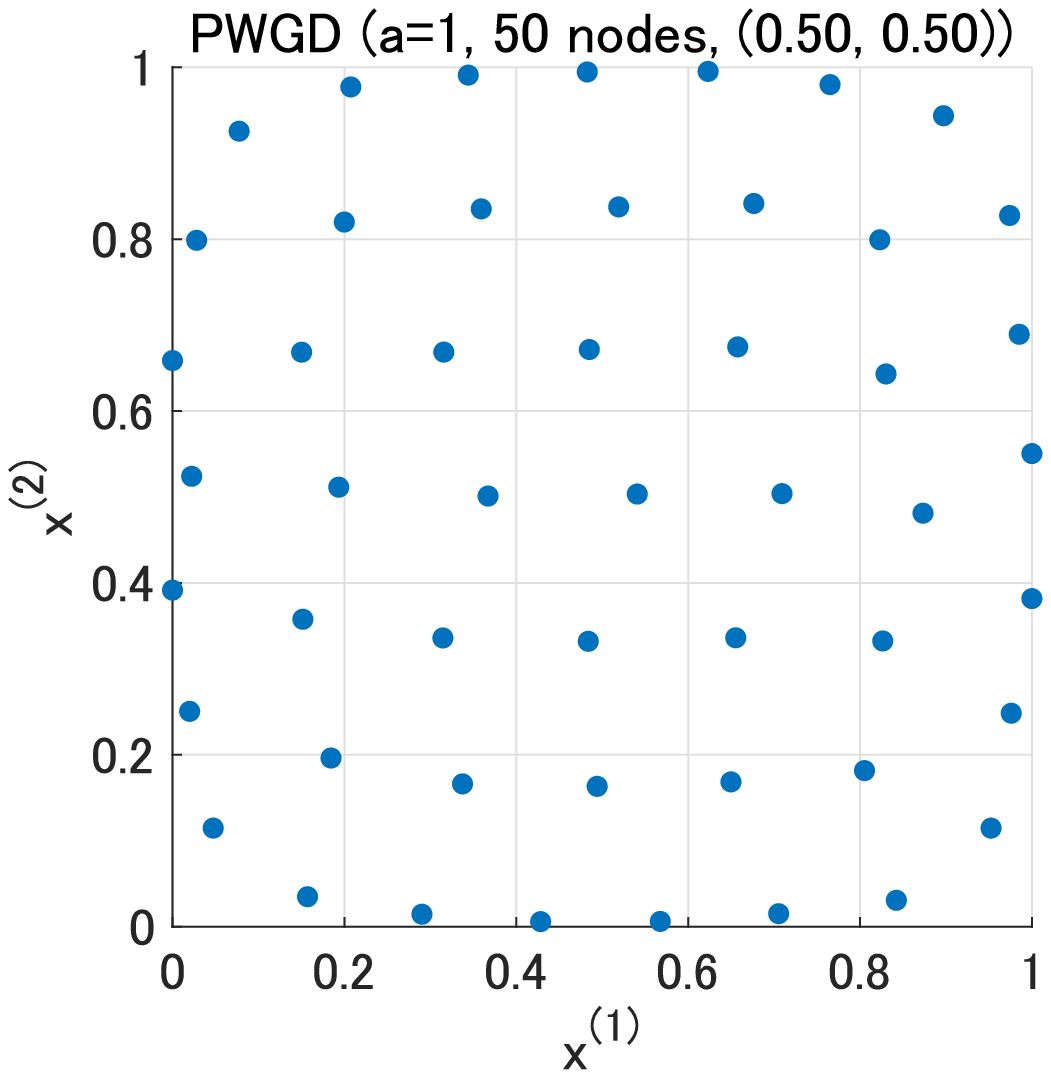}
\caption{$50$ points given by the proposed procedure with $(P,M) = (0.5, 0.5)$.}
\label{fig:2D_50p_PWGD_FS_050050}
\end{minipage}
\quad
\begin{minipage}[t]{0.49\linewidth}
\includegraphics[width = \linewidth]{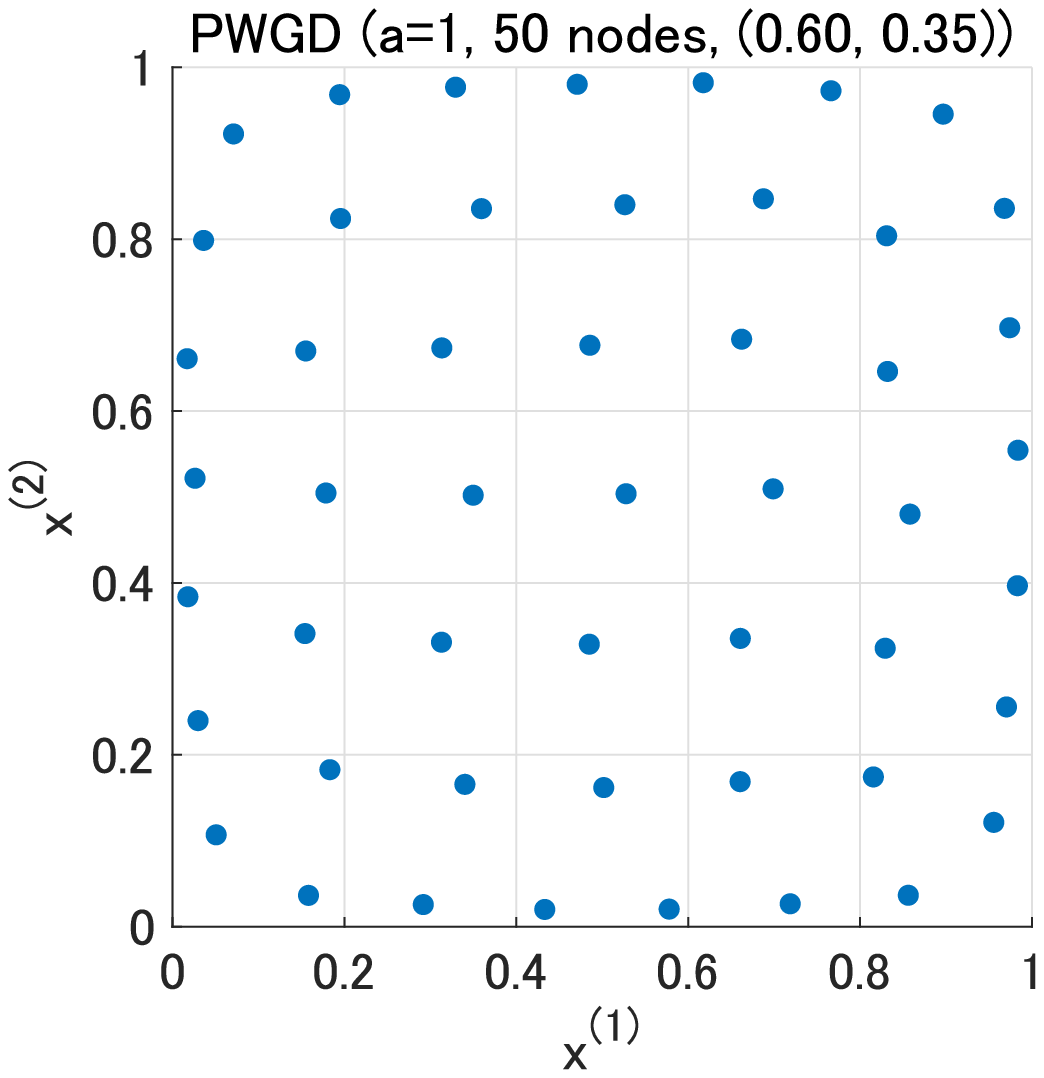}
\caption{$50$ points given by the proposed procedure with $(P,M) = (0.6, 0.35)$.}
\label{fig:2D_50p_PWGD_FS_060035}
\end{minipage}
\end{figure}

Next, 
in Figures~\ref{fig:2D_wce_P050} and~\ref{fig:2D_wce_P060}, 
we show the worst case errors 
for the points and weights given by methods M1, M2 and the proposed procedure. 
We can observe that 
the proposed procedure outperforms with the other methods
when the hyper-parameters $(P,M)$
are set appropriately according to $N$, the number of the points.

\begin{figure}[H]
\centering
\includegraphics[width = 0.6\linewidth]{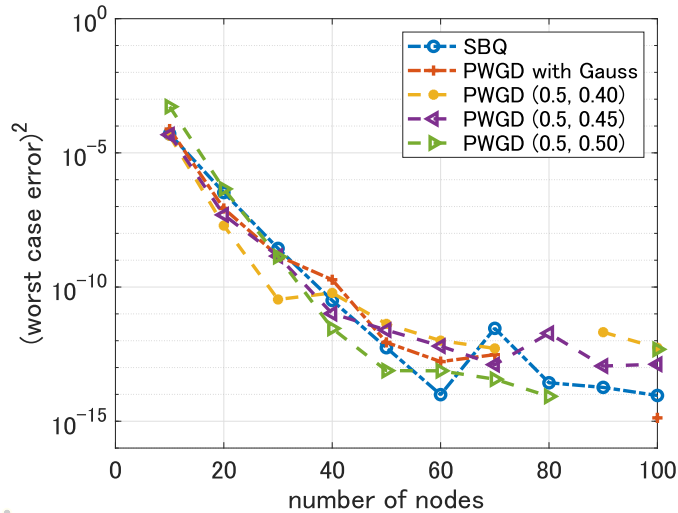}
\caption{Squared worst case errors. 
The horizontal axis corresponds to $N$. 
The legend ``PWGD with Gauss'' indicates method M2 and 
``PWGD $(P,M)$'' indicates the proposed procedure with the hyper-parameters $(P,M)$.
Here $P = 0.5$.}
\label{fig:2D_wce_P050}
\end{figure}

\begin{figure}[H]
\centering
\includegraphics[width = 0.6\linewidth]{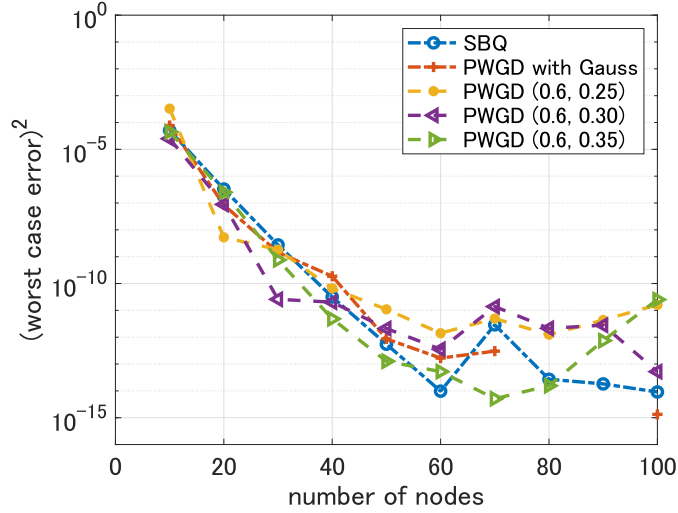}
\caption{Squared worst case errors.
The horizontal axis corresponds to $N$. 
The legend ``PWGD with Gauss'' indicates method M2 and 
``PWGD $(P,M)$'' indicates the proposed procedure with the hyper-parameters $(P,M)$.
Here $P = 0.6$.}
\label{fig:2D_wce_P060}
\end{figure}

%-----
\subsection{Three dimensional case: $d = 3$}

First, we show the generated points in
Figures~\ref{fig:3D_100p_SBQ}--\ref{fig:3D_100p_PWGD_FS_125012}. 
We can observe similar situations to the two-dimensional case. 

\begin{figure}[H]
\begin{minipage}[t]{0.49\linewidth}
\includegraphics[width = \linewidth]{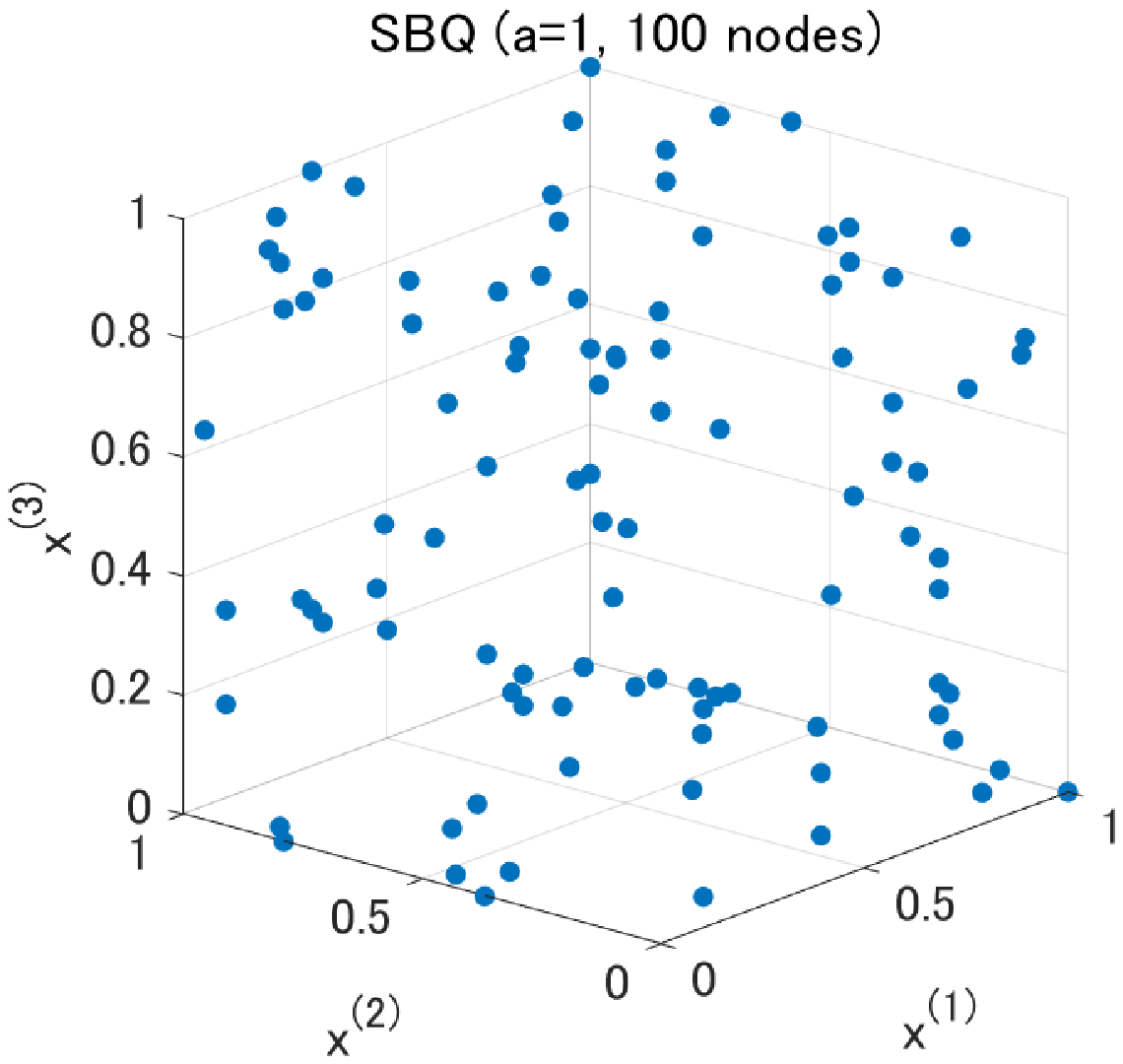}
\caption{$100$ points given by M1 (SBQ).}
\label{fig:3D_100p_SBQ}
\end{minipage}
\quad
\begin{minipage}[t]{0.49\linewidth}
\includegraphics[width = \linewidth]{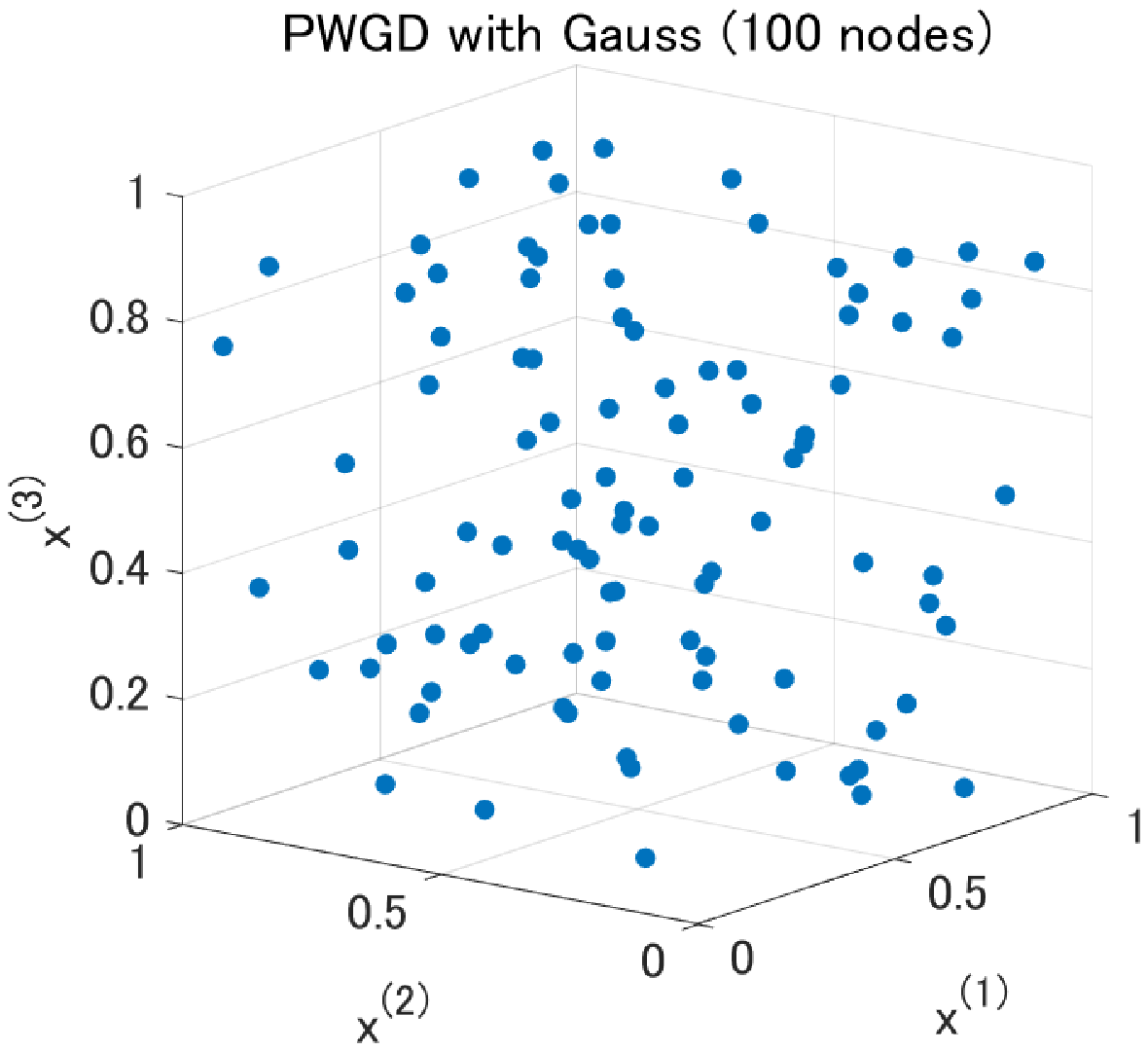}
\caption{$100$ points given by M2 (PWGD for the original worst-case error).}
\label{fig:3D_100p_PWGD_Gauss}
\end{minipage}
\end{figure}

\begin{figure}[H]
\begin{minipage}[t]{0.49\linewidth}
\includegraphics[width = \linewidth]{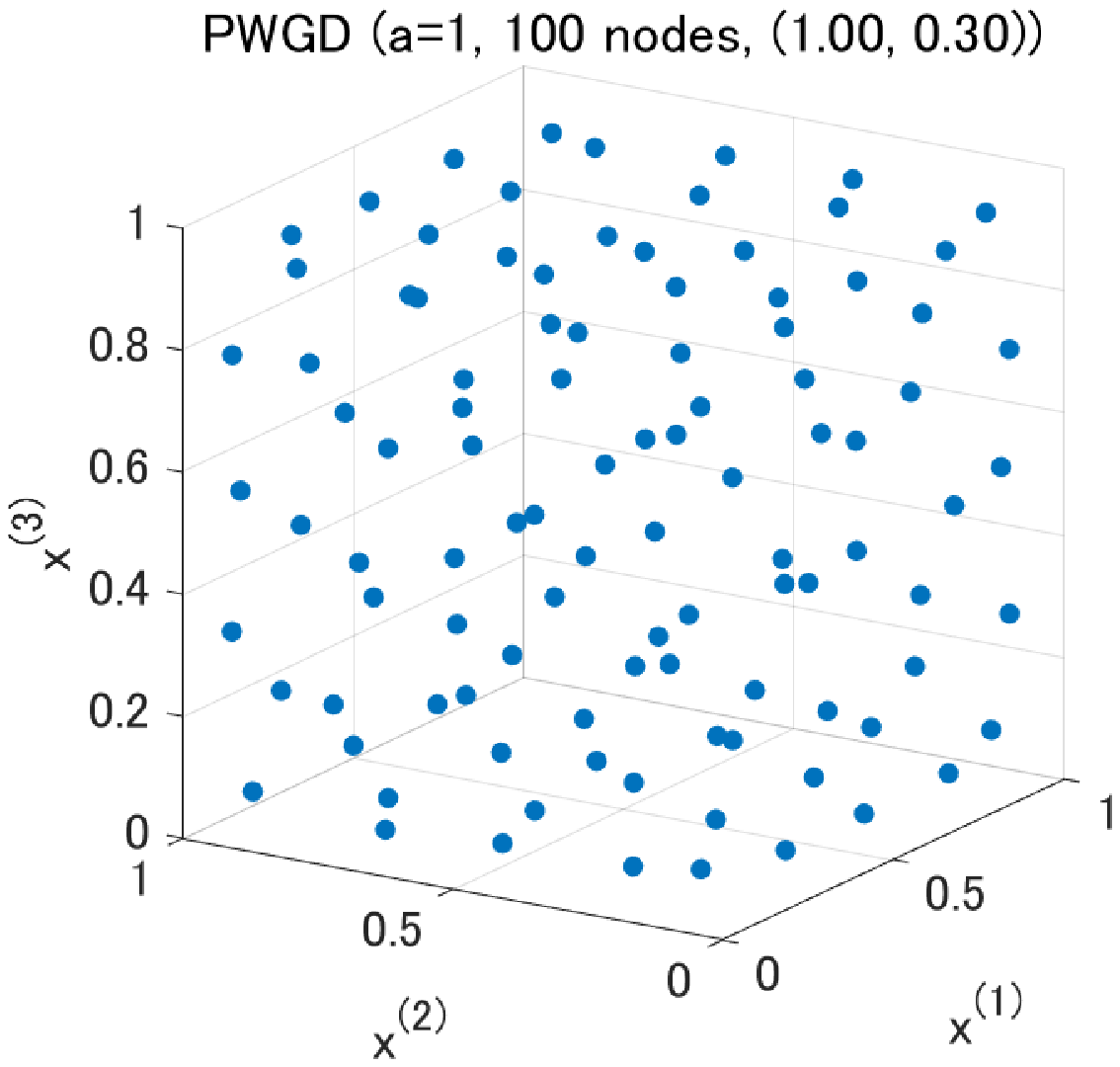}
\caption{$100$ points given by the proposed procedure with $(P,M) = (1.0, 0.3)$.}
\label{fig:3D_100p_PWGD_FS_100030}
\end{minipage}
\quad
\begin{minipage}[t]{0.49\linewidth}
\includegraphics[width = \linewidth]{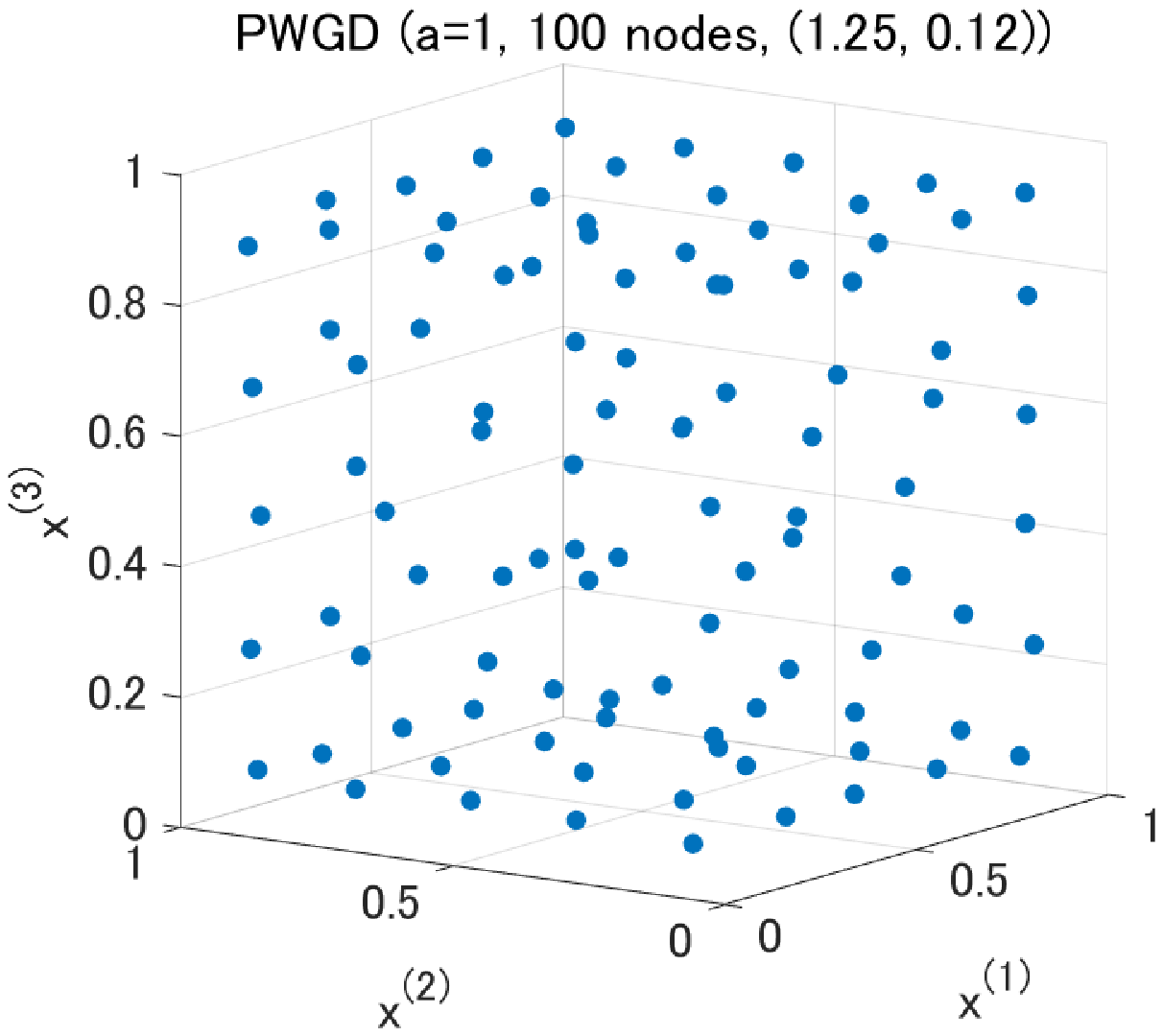}
\caption{$100$ points given by the proposed procedure with $(P,M) = (1.25, 0.12)$.}
\label{fig:3D_100p_PWGD_FS_125012}
\end{minipage}
\end{figure}

Next, 
in Figures~\ref{fig:3D_wce_P100} and~\ref{fig:3D_wce_P125},  
we show the worst case errors 
for the points and weights given by methods M1, M2 and the proposed procedure. 
We can observe that
the proposed procedure often outperforms method M2, 
which implies that the proposed function $I_{d} + R_{d}$ 
is well-suited to the point-wise gradient descent algorithm. 
On the other hand, 
the performance of the proposed procedure seems to be slightly worse than that of method M1 (SBQ), 
although the former outperforms the latter in some cases.

\begin{figure}[H]
\centering
\includegraphics[width = 0.6\linewidth]{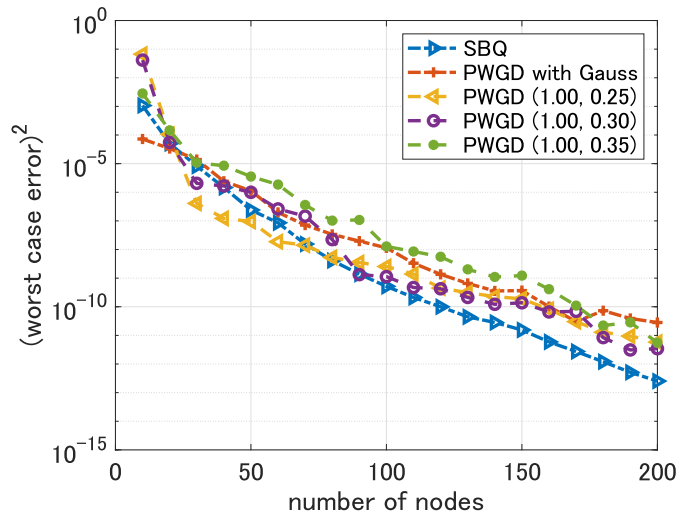}
\caption{Squared worst case errors.
The horizontal axis corresponds to $N$. 
The legend ``PWGD with Gauss'' indicates method M2 and 
``PWGD $(P,M)$'' indicates the proposed procedure with the hyper-parameters $(P,M)$.
Here $P = 1.0$.}
\label{fig:3D_wce_P100}
\end{figure}

\begin{figure}[H]
\centering
\includegraphics[width = 0.6\linewidth]{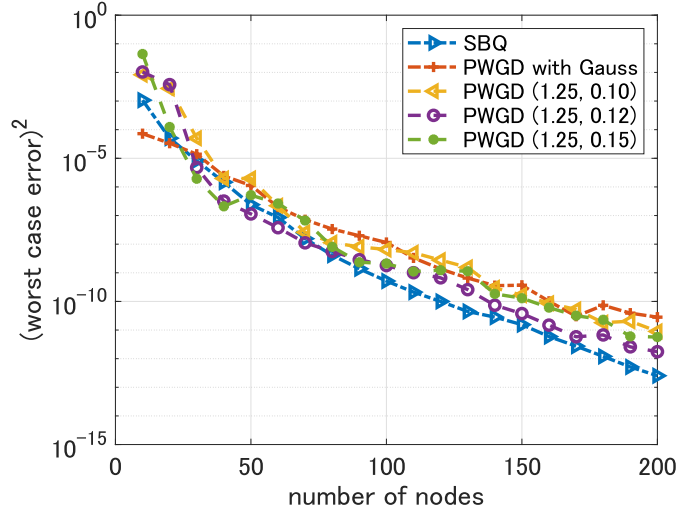}
\caption{Squared worst case errors. 
The horizontal axis corresponds to $N$. 
The legend ``PWGD with Gauss'' indicates method M2 and 
``PWGD $(P,M)$'' indicates the proposed procedure with the hyper-parameters $(P,M)$.
Here $P = 1.25$.}
\label{fig:3D_wce_P125}
\end{figure}

%--------------------
\section{Concluding remarks}
\label{sec:concl}

By using the fundamental solutions of the Laplacian on $\mathbf{R}$, 
we have provided the upper bound in~\eqref{eq:ub_WCE_two_terms} 
for the main terms in~\eqref{eq:WCE_two_terms} of the worst case error in an RKHSs with the Gaussian kernel. 
Based on the bound, 
we have proposed the objective functions with the regularization terms 
and Algorithm \ref{alg:PWGD} (PWGD) for generating points for quadrature in Section~\ref{sec:approx_min_wce}. 
Then, 
quadrature formulas are obtained by calculating the optimal weights with respect to the generated points. 
By the numerical experiments in Section~\ref{sec:num_exper}, 
we have observed that this procedure can outperform the SBQ and PWGD with the original worst case error
if the hyper-parameters are appropriate. 
We can guess that direct application of the PWGD to the original worst case error 
tends to make a set of points trapped in a bad local minimum.

We mention some topics for future work. 
Since 
the upper bound in~\eqref{eq:ub_WCE_two_terms} is loose 
and the regularization terms are artificial, 
it will be better to find discrete energies that is more suited to the PWGD. 
After that, theoretical guarantee of the performance of the PWGD should be required. 
Furthermore, we will consider generalization of such results to other kernels. 

%--------------------
\section*{Acknowledgements}

This work was partly supported by JST, PRESTO Grant Number JPMJPR2023, Japan.

%--------------------
%--------------------

%--------------------
%--------------------
\appendix

%--------------------
\section{Derivation of Formula~\eqref{eq:wce_optW}}
\label{sec:wce_optW}

Note that the right hand side of~\eqref{eq:WCE} is rewritten in the form
\begin{align}
\int_{\Omega} \int_{\Omega} K(x, y) \, \mathrm{d}x \mathrm{d}y
-2 \int_{\Omega} \boldsymbol{w}^{T} \, \boldsymbol{k}_{\mathcal{X}_{N}}(x) \, \mathrm{d}x
+ \boldsymbol{w}^{T} \, \mathcal{K}_{\mathcal{X}_{N}} \, \boldsymbol{w}
\label{eq:wce_RHKS_vec_mat}
\end{align}
by using $\boldsymbol{w} = (w_{1}, \ldots , w_{N})^{T}$. 
Then, 
by letting $\boldsymbol{w} = \boldsymbol{w}^{\ast}$ in~\eqref{eq:wce_RHKS_vec_mat}, 
we have
\begin{align}
& \int_{\Omega} \int_{\Omega} K(x, y) \, \mathrm{d}x \mathrm{d}y
-2 \int_{\Omega} \boldsymbol{w}^{T} \, \boldsymbol{k}_{\mathcal{X}_{N}}(x) \, \mathrm{d}x
+ \boldsymbol{w}^{T} \, \mathcal{K}_{\mathcal{X}_{N}}  \, \boldsymbol{w} \notag \\
& = 
\int_{\Omega} \int_{\Omega} K(x, y) \, \mathrm{d}x \mathrm{d}y
- \int_{\Omega} \int_{\Omega}
(\boldsymbol{k}_{\mathcal{X}_{N}}(x))^{T} \, \mathcal{K}_{\mathcal{X}_{N}} ^{-1}\, \boldsymbol{k}_{\mathcal{X}_{N}}(y)
\, \mathrm{d}x \, \mathrm{d}y \notag \\
& = 
\int_{\Omega} \int_{\Omega} K(x, y) \, \mathrm{d}x \mathrm{d}y
+ 
\frac{1}{\det \mathcal{K}_{\mathcal{X}_{N}} } \, 
\det \left[
\begin{array}{c|ccc}
0 & k_{1} & \cdots & k_{N} \\
\hline
k_{1} & & & \\
\vdots & & \mathcal{K}_{\mathcal{X}_{N}}  & \\
k_{N} & & & 
\end{array}
\right] \notag \\
& = 
\frac{1}{\det \mathcal{K}_{\mathcal{X}_{N}} } \, 
\det \left[
\begin{array}{c|ccc}
k_{0} & k_{1} & \cdots & k_{N} \\
\hline
k_{1} & & & \\
\vdots & & \mathcal{K}_{\mathcal{X}_{N}}  & \\
k_{N} & & & 
\end{array}
\right]. 
\notag
\end{align}

%--------------------
\section{Integrability of the fundamental solutions of the Laplacian}
\label{sec:int_fund_sol}

Regard the fundamental solution $G(x,y)$ in~\eqref{eq:fund_sol_Laplace} 
as a function of $x$ for a fixed $y$. 
Then, for $d \geq 2$, it has a singularity at $x = y$. 
Here we confirm that 
it is integrable on a bounded neighborhood of the singularity. 
In the following, 
we assume that $y = 0$ without loss of generality. 
Let $B_{d}[0,R] := \{ x \in \mathbf{R}^{d} \mid \| x \| \leq R \}$ be a closed ball. 

\subsection{$d = 2$}

We assume that $R \leq 1$. Then, we have
\begin{align}
\int_{B_{2}[0,R]} |G_{2}(x,0)| \, \mathrm{d}x
& = 
- \int_{0}^{R} r \, \mathrm{d}r 
\int_{0}^{2\pi} \mathrm{d}\theta \ 
\frac{1}{2\pi} \log r
= 
- \frac{R^{2}}{4} (2 \log R - 1) < \infty.
\notag
\end{align}

\subsection{$d \geq 3$}

We have
\begin{align}
\int_{B_{d}[0,R]} |G_{d}(x,0)| \, \mathrm{d}x
& = 
\int_{0}^{R} r^{d-1} \, \mathrm{d}r 
\cdot 
s_{d} 
\cdot  
\frac{1}{2 (d-2) s_{d}} \frac{1}{r^{d-2}}
= 
\frac{R^{2}}{4(d-2)} < \infty.
\notag
\end{align}

%--------------------
\section{Proofs}
\label{sec:proofs}

%-----
\begin{proof}[Proof of Lemma~\ref{lem:Green_exps_delta_expt_delta}]
By integration by parts and equality~\eqref{eq:Laplace_Green_delta}, 
we have
\begin{align}
\frac{\partial}{\partial s} \text{(LHS)}
& = 
\int_{\mathbf{R}^{d}}
\mathrm{d}y \
\mathrm{e}^{t \varDelta_{y}} \delta_{b}(y)
\int_{\mathbf{R}^{d}}
\mathrm{d}x \ 
G_{d}(x,y) 
\, \varDelta_{x} \mathrm{e}^{s \varDelta_{x}} \delta_{a}(x)
\notag \\
& = 
\int_{\mathbf{R}^{d}}
\mathrm{d}y \
\mathrm{e}^{t \varDelta_{y}} \delta_{b}(y)
\int_{\mathbf{R}^{d}}
\mathrm{d}x \ 
\varDelta_{x} G_{d}(x,y) 
\, \mathrm{e}^{s \varDelta_{x}} \delta_{a}(x)
\notag \\
& = 
\int_{\mathbf{R}^{d}}
\mathrm{d}y \
\mathrm{e}^{t \varDelta_{y}} \delta_{b}(y)
\int_{\mathbf{R}^{d}}
\mathrm{d}x \ 
\delta_{y}(x)
\, \mathrm{e}^{s \varDelta_{x}} \delta_{a}(x)
\notag \\
& = 
\int_{\mathbf{R}^{d}}
\mathrm{d}y \
\mathrm{e}^{t \varDelta_{y}} \delta_{b}(y) \ 
\mathrm{e}^{s \varDelta_{y}} \delta_{a}(y)
\notag \\
& = 
\int_{\mathbf{R}^{d}}
\mathrm{d}y \
\mathrm{e}^{(s+t) \varDelta_{y}} \delta_{b}(y) \ 
\delta_{a}(y),
\notag
\end{align}
and
\begin{align}
\frac{\partial}{\partial s} \text{(RHS)}
& = 
\int_{\mathbf{R}^{d}}
\mathrm{d}y \
G_{d}(a,y) 
\, \varDelta_{y} \mathrm{e}^{(s+t) \varDelta_{y}} \delta_{b}(y)
\notag \\
& = 
\int_{\mathbf{R}^{d}}
\mathrm{d}y \
\varDelta_{y} G_{d}(a,y) 
\, \mathrm{e}^{(s+t) \varDelta_{y}} \delta_{b}(y)
\notag \\
& = 
\int_{\mathbf{R}^{d}}
\mathrm{d}y \
\delta_{a}(y)
\, \mathrm{e}^{(s+t) \varDelta_{y}} \delta_{b}(y).
\notag
\end{align}
Therefore they coincide.
Furthermore, 
by taking the limit of both sides as $s \to +0$, 
we have
\begin{align}
\lim_{s \to +0} \text{(LHS)}
& = 
\int_{\mathbf{R}^{d}}
\mathrm{d}x
\int_{\mathbf{R}^{d}}
\mathrm{d}y \
G_{d}(x,y) 
\, \delta_{a}(x)
\, \mathrm{e}^{t \varDelta_{y}} \delta_{b}(y)
\notag \\
& = 
\int_{\mathbf{R}^{d}}
\mathrm{d}y \
\mathrm{e}^{t \varDelta_{y}} \delta_{b}(y)
\int_{\mathbf{R}^{d}}
\mathrm{d}x \
G_{d}(x,y) 
\, \delta_{a}(x)
\notag \\
& = 
\int_{\mathbf{R}^{d}}
\mathrm{d}y \
\mathrm{e}^{t \varDelta_{y}} \delta_{b}(y) \ 
G_{d}(a,y),
\notag 
\end{align}
and 
\begin{align}
\lim_{s \to +0} \text{(RHS)}
& = 
\int_{\mathbf{R}^{d}}
\mathrm{d}y \
G_{d}(a,y) 
\, \mathrm{e}^{t \varDelta_{y}} \delta_{b}(y).
\notag
\end{align}
Hence the conclusion holds. 
\end{proof}

%-----
\begin{proof}[Proof of Lemma~\ref{lem:Green_expt_delta_eq}]
Note that $\mathrm{e}^{t \varDelta} \delta_{b}$ is the heat kernel with center $b$:
\begin{align}
\mathrm{e}^{t \varDelta_{y}} \delta_{b}(y)
=
\frac{1}{(4\pi t)^{d/2}} \, \exp \left( -\frac{\| y - b \|^{2}}{4t} \right).
\label{eq:heat_kernel_on_Rd}
\end{align}
Therefore the function $\mathrm{e}^{t \varDelta_{y}} \delta_{b}(y)$ 
depends only on $t$ and the difference $y - b$. 
In addition, 
the function $G_{d}(b,y)$ depends only on $y-b$ 
as shown by formula~\eqref{eq:fund_sol_Laplace}. 
Then, 
by integration by substitution with $z = y - b$, 
we have
\begin{align}
\int_{\mathbf{R}^{d}}
\mathrm{d}y \
G_{d}(b,y) 
\, \mathrm{e}^{t \varDelta_{y}} \delta_{b}(y)
& =
\int_{\mathbf{R}^{d}}
\mathrm{d}z \
G_{d}(0,z) 
\, \mathrm{e}^{t \varDelta_{z}} \delta_{0}(z)
\notag \\
& = 
\frac{1}{(4\pi t)^{d/2}} 
\int_{\mathbf{R}^{d}}
\mathrm{d}z \ 
G_{d}(z, 0) \, 
\exp \left( -\frac{\| z \|^{2}}{4t} \right). 
\notag 
\end{align}
Clearly this value is independent of $b$. 
Furthermore, 
to show the boundedness of this value, 
we use the following estimate:
\begin{align}
& \left|
\int_{\mathbf{R}^{d}}
\mathrm{d}y \
G_{d}(b,y) 
\, \mathrm{e}^{t \varDelta_{y}} \delta_{b}(y)
\right|
\notag \\
& \leq 
\frac{1}{(4\pi t)^{d/2}} 
\left(
\int_{B_{d}[0,1]} + \int_{\mathbf{R}^{d} \setminus B_{d}[0,1]} 
\right)
\mathrm{d}z \ 
|G_{d}(z, 0)| \, 
\exp \left( -\frac{\| z \|^{2}}{4t} \right)
\notag \\
& \leq 
\frac{1}{(4\pi t)^{d/2}} 
\left(
\int_{B_{d}[0,1]} \mathrm{d}z \ 
|G_{d}(z, 0)| 
+ 
\int_{\mathbf{R}^{d} \setminus B_{d}[0,1]} \mathrm{d}z \ 
|G_{d}(z, 0)| \, 
\exp \left( -\frac{\| z \|^{2}}{4t} \right)
\right). 
\label{eq:fund_sol_exp_delta_abs}
\end{align}
The first term in the parenthesis is bounded because of the argument in Section~\ref{sec:int_fund_sol}. 
To estimate the second term, we note that 
$|G_{d}(z,0)|$ depends only on $\| z \|$ and bounded by $c_{d} \| z \|$ for some constant $c_{d} > 0$ when $\| z \| \geq 1$. 
Therefore we have
\begin{align}
\int_{\mathbf{R}^{d} \setminus B_{d}[0,1]} \mathrm{d}z \ 
|G_{d}(z, 0)| \, 
\exp \left( -\frac{\| z \|^{2}}{4t} \right)
& \leq s_{d} \int_{1}^{\infty} r^{d-1} \, \mathrm{d}r \cdot c_{d} \, r \cdot \exp \left( -\frac{r^{2}}{4t} \right)
\notag \\
& = 
s_{d} c_{d} \int_{1}^{\infty} r^{d} \exp \left( -\frac{r^{2}}{4t} \right) \, \mathrm{d}r 
< \infty.
\notag
\end{align}
From these, 
the RHS of~\eqref{eq:fund_sol_exp_delta_abs} is bounded. 
\end{proof}

%-----
\begin{proof}[Proof of Lemma~\ref{lem:Green_expt_delta_disj}]
When $t \to +0$, 
the LHS tends to $G_{d}(a,b)$.
By differentiating the LHS, we have
\begin{align}
\frac{\partial}{\partial t} \text{(LHS)}
& = 
\int_{\mathbf{R}^{d}}
\mathrm{d}y \
G_{d}(a,y) 
\, \varDelta_{y} \mathrm{e}^{t \varDelta_{y}} \delta_{b}(y)
\notag \\
& =
\int_{\mathbf{R}^{d}}
\mathrm{d}y \
\varDelta_{y} G_{d}(a,y) 
\, \mathrm{e}^{t \varDelta_{y}} \delta_{b}(y)
\notag \\
& =
\int_{\mathbf{R}^{d}}
\mathrm{d}y \
\delta_{a}(y)
\, \mathrm{e}^{t \varDelta_{y}} \delta_{b}(y)
\notag \\
& = 
\left( \mathrm{e}^{t \varDelta} \delta_{b} \right) \! (a)
\notag \\
& = 
\frac{1}{(4\pi t)^{d/2}} \, \exp \left( -\frac{\| a - b \|^{2}}{4t} \right), 
\notag
\end{align}
where we use formula~\eqref{eq:heat_kernel_on_Rd} in the last equality. 
Then, the conclusion follows. 
\end{proof}

\begin{proof}[Proof of Lemma~\ref{thm:ineq_int_heat_fund_sol}]
Set $\alpha := \| x - y \|$, 
which satisfies $0 < \alpha \leq D$. 
For preparation, 
we define a function $g$ by
\begin{align}
g(s) 
:= 
\frac{1}{s^{d/2}} \, \exp \left( -\frac{\alpha^{2}}{4s} \right) 
\notag
\end{align}
for $s > 0$. 
Since 
\begin{align}
g'(s) = \frac{1}{4s^{d/2+2}} \, \exp \left( -\frac{\alpha^{2}}{4s} \right) \, (\alpha^{2} - 2ds), 
\notag
\end{align}
the function $g$ is unimodal and becomes maximum at 
\(
\displaystyle
s = \dfrac{\alpha^{2}}{2d}
\).
Then, if we set
\(
\displaystyle
s_{\ast} = \dfrac{D^{2}}{2d}
\), 
the function $g$ is monotone decreasing on $[s_{\ast}, \infty)$. 
Therefore we have
\begin{align}
g(s) \geq g(2s_{\ast}) \ 
\text{for any $s$ with $s_{\ast} \leq s \leq 2s_{\ast}$}. 
\notag
\end{align}
By using this inequality, 
we can derive the following estimate: 
\begin{align}
& \int_{0}^{t} \frac{1}{(4\pi s)^{d/2}} \, \exp \left( -\frac{\alpha^{2}}{4s} \right) \, \mathrm{d}s
\notag \\
& \geq 
\frac{1}{(4\pi)^{d/2}} 
\left( \int_{s_{\ast}}^{2s_{\ast}} + \int_{2s_{\ast}}^{t}\right) 
g(s) \, \mathrm{d}s
\notag \\
& \geq 
\frac{1}{(4\pi)^{d/2}} 
\left(
s_{\ast} \, g(2s_{\ast}) 
+ 
\exp \left( -\frac{\alpha^{2}}{8s_{\ast}} \right) 
\int_{2s_{\ast}}^{t} \frac{1}{s^{d/2}} \, \mathrm{d}s
\right)
\notag \\
& \geq 
\frac{1}{(4\pi)^{d/2}} 
\left(
s_{\ast} \, g(2s_{\ast}) 
+ 
\exp \left( -\frac{D^{2}}{8s_{\ast}} \right) 
\int_{2s_{\ast}}^{t} \frac{1}{s^{d/2}} \, \mathrm{d}s
\right)
\notag \\
& =
\begin{cases}
\displaystyle
\frac{1}{(4\pi)^{d/2}} 
\left[
\frac{d^{d/2-1}}{2D^{d-2}} \, \exp \left( -\frac{d \alpha^{2}}{4D^{2}} \right) 
+
\frac{\mathrm{e}^{-d/4}}{1-d/2}
\left(
\frac{1}{t^{d/2-1}} - \frac{d^{d/2-1}}{D^{d-2}}
\right)
\right] & (d \neq 2), \\[12pt]
\displaystyle
\frac{1}{(4\pi)^{d/2}} 
\left[
\frac{d^{d/2-1}}{2D^{d-2}} \, \exp \left( -\frac{d \alpha^{2}}{4D^{2}} \right) 
+
\mathrm{e}^{-d/4}
\log \left(
\frac{td}{D^{2}}
\right)
\right] & (d = 2). \\[6pt]
\end{cases}
\end{align}
Thus the conclusion holds.
\end{proof}

\end{document}